\documentclass[a4paper,10pt]{amsart}
\usepackage[utf8]{inputenc}
\usepackage{amsthm}
\usepackage{amsmath}
\usepackage{bm}
\usepackage{amsfonts}
\usepackage{amssymb}
\usepackage{mathtools}
\usepackage{mathrsfs}
\usepackage{setspace}
\usepackage{xcolor}
\usepackage{textgreek}
\usepackage{todonotes}
\usepackage{scalerel}
\usepackage[a4paper, left=3cm, right=3cm, top=3cm, bottom=3cm]{geometry}
\usepackage{thmtools} 
\usepackage{bm}

\usepackage[pdfdisplaydoctitle,colorlinks,breaklinks,urlcolor=blue,linkcolor=blue,citecolor=blue]{hyperref} 

\newtheorem{thm}{Theorem}[section]

\newtheorem{lem}[thm]{Lemma}

\newtheorem{prop}[thm]{Proposition}

\newcommand{\N}{\mathbb{N}}

\newcommand{\R}{\mathbb{R}}

\newcommand{\T}{\mathbb{T}}

\newcommand{\dvg}{\mathord{{\rm div}}\,}
\newcommand{\dvgphi}{\mathord{{\rm div}}^\phi}
\newcommand{\dvgphien}{\mathord{{\rm div}}^{\phi_n}}
\newcommand{\dvgphin}{\mathord{{\rm div}}^{\phi_{n+1}}}\newcommand{\Dphi}{\Delta^{\hspace{-0.05cm}\phi}}
\newcommand{\Dphien}{\Delta^{\hspace{-0.05cm}\phi_n}}
\newcommand{\Dphin}{\Delta^{\hspace{-0.05cm}\phi_{n+1}}}
\newcommand{\nablaphi}{\nabla^{\phi}}
\newcommand{\nablaphien}{\nabla^{\phi_n}}
\newcommand{\nablaphin}{\nabla^{\phi_{n+1}}}
\newcommand{\curl}{\mathord{{\rm curl}}\,}

\newcommand{\curlphin}{\mathord{{\rm curl}}^{\phi_{n+1}}}
 
\newcommand\scslash{\stretchrel*{$/$}{\textsc{e}}}
\newcommand\rpar{r_{\hspace{-0.05cm} \scslash \hspace{-0.1cm} \scslash}}
\theoremstyle{remark}
\newtheorem{rmk}[thm]{Remark}

\numberwithin{equation}{section}

\title[Ill-posedness Navier-Stokes transport noise]{Global existence and non-uniqueness for the Cauchy problem associated to 3D Navier-Stokes equations perturbed by transport noise}

\author[U. Pappalettera]{Umberto Pappalettera}
  \address{Fakult\"at f\"ur Mathematik, Universit\"at Bielefeld, D-33501 Bielefeld, Germany}
  \email{upappale(at)math.uni-bielefeld.de}
  
\keywords{Navier-Stokes equations, transport noise, convex integration}
\date\today

\begin{document}

\begin{abstract}
We show global existence and non-uniqueness of probabilistically strong, analytically weak solutions of the three-dimensional Navier-Stokes equations perturbed by Stratonovich transport noise.
We can prescribe either: \emph{i}) any divergence-free, square integrable intial condition; or \emph{ii}) the kinetic energy of solutions up to a stopping time, which can be chosen arbitrarily large with high probability. 
Solutions enjoy some Sobolev regularity in space but are not Leray-Hopf.
\end{abstract}

\maketitle

\section{Introduction}

In the present note we consider the Navier-Stokes equations on the three-dimensional torus $\T^3=(\R/2\pi)^3$ perturbed by Stratonovich transport noise:
\begin{align} \label{eq:NS}
\begin{cases}
d u
+ (u \cdot \nabla) u \, dt 
+ \sum_{k \in I} (\sigma_k \cdot \nabla) u \bullet dB^k
+ \nabla p \, dt
= \Delta u \, dt,
\\
\dvg u = 0.
\end{cases}
\end{align}
The noise is characterized by a finite collections of smooth divergence-free vector fields $\{\sigma_k\}_{k \in I}$ and i.i.d. standard Brownian motions $\{B^k\}_{k \in I}$ on a given filtered probability space $(\Omega,\mathcal{F},\{\mathcal{F}_t\}_{t \geq 0}, \mathbb{P})$ satisfying the usual conditions.

A \emph{probabilistically strong, analytically weak} solution to \eqref{eq:NS} is a stochastic process $u$, progressively measurable with respect to $\{\mathcal{F}_t\}_{t\geq 0}$, such that \eqref{eq:NS} holds true when integrated in time and space against a smooth, divergence free, compactly supported test function. Notice that the pressure $p$ is uniquely determined by $u$ and the relations (assuming each term makes sense at least as a distribution)
\begin{align*}
\dvg (u \cdot \nabla) u \, dt 
+ 
\sum_{k \in I} \dvg (\sigma_k \cdot \nabla) u \bullet dB^k
+ 
\Delta p \, dt
=
0,
\quad
\int_{\T^3} p = 0.
\end{align*}

Hereafter, for $p \in [1,\infty]$ we denote $L^p_x := L^p(\T^3,\R^n)$, $n \in \N$ depending on the context. A similar convention will be in force for other spaces of functions on $\T^3$ considered in the following, as the Sobolev spaces $W^{s,p}_x$, $s \in \R$, and $H^s_x := W^{s,2}_x$.
Our goal is to prove global existence and non-uniqueness of probabilistically strong, analytically weak solutions to \eqref{eq:NS} emanating from any prescribed divergence-free initial condition $u_0 \in L^2_x$. 
In particular, we can prove the following:

\begin{thm} \label{thm:u_0}
Given any divergence-free $u_0 \in L^2_x$ almost surely and independent of the Brownian motions $\{B^k\}_{k \in I}$, there exist infinitely many probabilistically strong, analytically weak solutions of \eqref{eq:NS} of class $u \in L^p_{loc}([0,\infty), L^2_x) \cap L^2_{loc}([0,\infty), W^{1,1}_x) \cap C_{loc}([0,\infty), H^{-1}_x)$ for every {$p \in [1,4)$} almost surely with initial condition $u|_{t=0} = u_0$.
\end{thm}

If we don't require $u|_{t=0}$ to attain a particular initial datum $u_0$, then we can produce slightly more regular solutions and even prescribe the kinetic energy thereof up to an arbitrary large stopping time.

\begin{thm} \label{thm:energy}
There exists $\gamma>0$ with the following property.
For every $\varkappa \in (0,1)$ and $T \in (0,\infty)$ there exists a stopping time $\mathfrak{t}$ such that $\mathbb{P} \{ \mathfrak{t} \geq T \} \geq \varkappa$ and, for every energy profile $e$ on $[0,\infty)$ satisfying $\underline{e} := \inf_{t \geq 0} e(t) > 0$ and $\overline{e} := \| e \|_{C^1_t} < \infty$, there exists a probabilistically strong, analytically weak solution of \eqref{eq:NS} of class $u \in C_{loc}([0,\infty),H^{\gamma}_x \cap W^{1+\gamma,1}_x)$ almost surely with kinetic energy 
\begin{align*}
\int_{\T^3} |u(x,t)|^2 dx = e(t),
\quad
\forall t \in [0,\mathfrak{t}].
\end{align*}
Moreover, if $u_1$, $u_2$ are the solutions associated with kinetic energies $e_1$, $e_2$ and $e_1(t)=e_2(t)$ for every $t \in [0,T/2]$, then $u_1(x,t)=u_2(x,t)$ for every $x \in \T^3$ and $t \in [0,T/2]$. 
\end{thm}

Our proofs strongly rely on convex integration techniques developed in previous works, both deterministic and stochastic.
For deterministic Navier-Stokes, solutions with prescribed energy profile were constructed in \cite{BuVi19,BuVi19b} using intermittent Beltrami flows.
Then, using intemittent jets as building blocks (first introduced in \cite{BuCoVi18}), \cite{BuMoSz20} proved non-uniqueness results for (deterministic) power-law fluids with prescribed initial condition.
In \cite{HoZhZh22b+} these results are extended to Navier-Stokes equations perturbed with additive noise. The main difficulty overcome by \cite{HoZhZh22b+} consists in controlling the interplay between the noise and the convex integration scheme, via the introduction of a sequence of stopping times and gluing solutions on different time intervals.
More recently, \cite{HoLaPa22+} applied for the first time a convex integration scheme to an equation perturbed by noise of transport type, thanks to a flow transformation: roughly speaking, the stochastic PDE perturbed by transport noise is reduced to a random PDE (that is, a partial differential equation with random coefficients but without stochastic integrals) and the latter is solved by a suitably designed convex integration scheme.
In particular, in \cite{HoLaPa22+} the authors considered Euler equations.
All these works have significantly influenced the present one. 

The main novelties contained in this note, apart from the results here presented, are mainly of technical nature. Particularly non-trivial is the fact that, contrary to previous convex integration schemes for Navier-Stokes equations, here we need to include iterative estimates in $L^1_x$ for the pressure.  
Those estimates are not deduced directly from the expression of the perturbation, but rather require \emph{ad hoc} arguments based on the introduction of an auxiliary iterative scheme, see \autoref{ssec:oscillation} and \autoref{ssec:pressure_est} for details.   
{Also, worth of mention is the bound $p<4$ in \autoref{thm:u_0}, due to the lack of uniform-in-time control over the $L^1_x$ norm of the Reynolds stress.}

To close this introduction, let us mention the main motivations behind our work.
In recent years, Stratonovich transport noise has shown regularizing properties (in the sense of improved well-posedness) in transport equation \cite{FlGuPr10}, point vortices dynamics \cite{FlGuPr11}, Vlasov-Poisson equations \cite{DeFlVi14}, Navier-Stokes equations in vorticity from \cite{FlLu21}, and related models \cite{FlGaLu21c,Lu21++,Ag22+,La22+}. 
Furthermore, its appearance in fluid dynamics equations is motivated by multiscale arguments \cite{FlPa21,FlPa22,DePa22+}.

Our results, however, rule out the possibility that Navier-Stokes equations in velocity form be regularized by transport noise (at least in the sense of restoring uniqueness; delay or prevention of blow-up of regular solutions remain plausible).

Moreover, we would like to point out that the mere global existence of probabilistically strong solutions to the Cauchy problem associated to Navier-Stokes equations in dimension three was unknown before this work. 
However, our solutions do not enjoy the desirable $L^\infty([0,\infty),L^2_x) \cap L^2([0,\infty),H^1_x)$ regularity of Leray-Hopf weak solutions.
Whether probabilistically strong solutions -- either Leray-Hopf or not -- can be also obtained without the use of convex integration techniques remains an important open problem.  

\section*{Acknowledgements}
This project has received funding from the European Research Council (ERC) under the European Union’s Horizon 2020 research and innovation programme (grant agreement No. 949981).

\section{Preliminaries} \label{sec:prelim}

\subsection{Flow transformation and reduction to a random PDE}

As already discussed in \cite{HoLaPa22+}, applying the flow transformation
\begin{align*}
v(x,t) &:= u(\phi(x,t),t), 
\quad
q(x,t) := p(\phi(x,t),t),
\end{align*}
where we denote by $\phi:\T^3 \times [0,\infty) \to \T^3$ the flow of measure preserving diffeomorphisms of the torus given by 
\begin{align} \label{eq:flow}
\phi(x,t) := x + \sum_{k \in I}\int_0^t \sigma_k(\phi(x,s)) \bullet dB^k(s),
\end{align}
it is possible to rewrite the SPDE \eqref{eq:NS} as a random PDE:
\begin{align} \label{eq:random_NS}
\begin{cases}
\partial_t v 
+ \dvgphi (v \otimes v)
+ \nablaphi q
= \Dphi v,
\\
\dvgphi v = 0 .
\end{cases}
\end{align}
Here $v \otimes v$ denotes the tensor product and the symbols $\dvgphi$, $\nablaphi$, $\Dphi$ are abbreviations for the space-time dependent differential operators
\begin{align*}
\dvgphi v := [\dvg(v \circ \phi^{-1})] \circ \phi,
\quad
\nablaphi q := [\nabla(q \circ \phi^{-1})] \circ \phi,
\quad
\Dphi v := [\Delta(v \circ \phi^{-1})] \circ \phi.
\end{align*} 
By the same arguments of \cite[Proposition A.1]{HoLaPa22+} one can see that systems \eqref{eq:NS} and \eqref{eq:random_NS} are equivalent, namely $u$ is a probabilistically strong, analytically weak solution of \eqref{eq:NS} if and only if $v$ is a probabilistically strong, analytically weak solution of \eqref{eq:random_NS}. We point out that in \eqref{eq:NS} and \eqref{eq:random_NS} the pressure terms $p$ and $q$ are always implicitly reconstructed from the velocity fields (assuming for instance null space average) and do not appear in the distributional formulation of the equations, when integrating against smooth test functions $h$ satisfying $\dvg h = 0$ and $\dvgphi h = 0$, respectively.
Thus, noticing also that $\int_{\T^3} |u(x,t)|^2 dx = \int_{\T^3} |v(x,t)|^2 dx$ for every $t$ since $\phi$ is measure preserving, it is sufficient to prove \autoref{thm:u_0} and \autoref{thm:energy} for solutions $v$ of \eqref{eq:random_NS}.

\subsection{Mollification of the noise}

Without loss of generality we assume that every realisation of the $\R^{|I|}$-valued driving noise $B=(B^k)_{k \in I}$ has H\"older $C^{1/2-}_{loc}$ time regularity.
Since we need smoothness of the flow during the construction, we can argue as in \cite{HoLaPa22+} and introduce a smooth mollifier $\vartheta:\R \to \R$ with support contained in $(0,1)$, and define for $t \in \R$ and some parameter $\varsigma_n>0$, $n \in \N$ to be properly chosen below:
\begin{align*}
\vartheta_n(t) \coloneqq\varsigma_n^{-1} \vartheta(t\varsigma_n^{-1}),
\quad
B_n(t) \coloneqq (B \ast \vartheta_n) (t)
=
\int_\R B(t-s) \vartheta_n(s) ds,
\end{align*}
with the convention that $B(t-s)=B(0)=0$ whenever $t-s$ is negative. 
Notice that $B_n$ is smooth at every time $t \in \R$ and it is identically zero for negatives times, being the support of $\vartheta$ contained in $(0,1)$. 
Finally, define $\phi_n$ as the unique solution of the integral equation
\begin{align} \label{eq:def_phi_n}
\phi_n(x,t) \coloneqq x + \sum_{k \in I}\int_0^t \sigma_k(\phi_n(x,s)) \,dB^k_n(s),
\quad
x \in \T^3, \,t \in \R.
\end{align}
Notice that $\phi_n(x,t)= \phi_n(x,0) = x$ for $t<0$, and for every fixed $t \in \R$ the map $\phi_n(\cdot,t)$ is measure preserving.
We extend the flow $\phi$ defined by \eqref{eq:flow} to $t<0$ similarly.

For any $r \geq 0$, stopping time $\mathfrak{t}$ and Banach space $E$, we denote $C^r_\mathfrak{t} E := C^r((-\infty,\mathfrak{t}],E)$.
The following is a consequence of \cite[Lemma 2.2]{HoLaPa22+}.
\begin{lem} \label{lem:flow}
Fix $T \in (0,\infty)$, $\varkappa \in (0,1)$, $\beta \in [0,1/4]$ and $\kappa,r \in \N$.
Then there exist a constant $C$, a stopping time $\mathfrak{t}$ such that $\mathbb{P} \{ \mathfrak{t} \geq T \} \geq \varkappa$, and a sequence of stopping times $\mathfrak{t}=\mathfrak{t}_1 \leq ... \leq \mathfrak{t}_L \leq ...$ such that $\mathfrak{t}_L \to \infty$ almost surely as $L \to \infty$ and for every $L,n \in \N$, $L \geq 1$ the following hold
\begin{gather*}
\| \phi_{n+1} - \phi_n \|_{C^\beta_{\mathfrak{t}_L} C^\kappa_x}
\leq CL (n+1) \varsigma_n^{1/4-\beta},
\quad
\| \phi_n \|_{C^{1/4}_{\mathfrak{t}_L} C^\kappa_x}
\leq CL,
\\
\| \phi_{n+1}^{-1} - \phi_n^{-1} \|_{C^\beta_{\mathfrak{t}_L} C^\kappa_x}
\leq CL (n+1) \varsigma_n^{1/4-\beta}, 
\quad
\| \phi_n^{-1} \|_{C^{1/4}_{\mathfrak{t}_L} C^\kappa_x}
\leq CL,
\end{gather*}
as well as
\begin{align*}
\| \phi_n \|_{C^r_{\mathfrak{t}_L} C^\kappa_x} 
\leq CL
\varsigma_n^{{1/4}-r},
\quad
\| \phi_n^{-1} \|_{C^r_{\mathfrak{t}_L} C^\kappa_x} 
\leq CL 
\varsigma_n^{{1/4}-r}.
\end{align*}
\end{lem}
\begin{rmk}
In the previous lemma we are in fact only assuming that $B$ has $H$-H\"older trajectories for some $H>1/4$. This is coherent with the fact that we can replace the noise $B$ in \eqref{eq:NS} with a fractional Brownian motion $B^H$ with Hurst parameter $H>1/4$, see also \cite[Remark 2.3]{HoLaPa22+}.
\end{rmk}

\begin{rmk}[On globality-in-time of solutions]
{
Both \autoref{thm:u_0} and \autoref{thm:energy} state the existence of global-in-time solutions to \eqref{eq:NS}, which we obtain by construction.
Indeed, our convex integration scheme relies on perturbations defined for every time $t>0$.
Then we use the stopping times $\mathfrak{t}_L$ just to give good bounds on suitable norms of the solutions up to time $\mathfrak{t}_L$, depending in $L$ as $L^{\mathfrak{m}}$, $\mathfrak{m} :=m^{m(1+\log\log L)}$ for some integer $m$ independent of $L$. This approach was previously pursued in \cite{HoLaPa22+} for Euler equations.
However, for the sake of simplicity we prefer to work only on the time interval $[0,\mathfrak{t}]$, omitting the verifications on $[0,\mathfrak{t}_L]$ for $L>1$.
}
\end{rmk}

\subsection{Intermittent jets} \label{ssec:jets}

The main building blocks of our convex integration scheme are the so-called \emph{intermittent jets}, first introduced in \cite{BuCoVi18}, cf. also \cite[Section 7.4]{BuVi19b}.

Recall \cite[Lemma 6.6]{BuVi19b} (see also \cite[Lemma 2.4]{DaSz17}), according to which there exists a finite set $\Lambda\subset \mathbb{S}^2\cap \mathbb{Q}^3$ such that for each $\xi\in \Lambda$ there exists a  $C^\infty$-function $\gamma_\xi:\overline{B_{1/2}}(\mathrm{Id})\rightarrow\mathbb{R}$ such that
\begin{equation*}
R=\sum_{\xi\in\Lambda}\gamma_\xi^2(R)\, \xi\otimes \xi,
\quad
\forall R \in \overline{B_{1/2}}(\mathrm{Id}).
\end{equation*}
Here $\mathbb{S}^2\cap \mathbb{Q}^3$ is the set of points of the unit sphere in $\R^3$ with rational coordinates, and $\overline{B_{1/2}}(\mathrm{Id}) \subset \mathcal{S}^{3\times 3}$ denotes the closed ball of radius $1/2$ around the identity matrix $\mathrm{Id}$, in the space of $3 \times 3$ symmetric matrices.

For each $\xi\in \Lambda$ let $A_\xi\in \mathbb{S}^2\cap \mathbb{Q}^3$ be an orthogonal vector to $\xi$. Then for each $\xi\in\Lambda$ we have that $\{\xi, A_\xi, \xi\times A_\xi\}\subset \mathbb{S}^2\cap \mathbb{Q}^3$ is an orthonormal basis of $\mathbb{R}^3$.
We label by $n_*$ the smallest positive natural such that $\{n_*\xi, n_*A_\xi, n_*\xi\times A_\xi\}\subset \mathbb{Z}^3$ for every  $\xi\in \Lambda$.

Next, let $\Theta:\R^2 \to \R$ and $\psi:\mathbb{R}\rightarrow\mathbb{R}$ be smooth functions with support in the ball of radius $1$ (of the respective domains), renormalized such that $\theta :=-\Delta\Theta$ and $\psi$ obey $\| \theta \|_{L^2} = 2\pi$ and $\|\psi \|_{L^2} = (2\pi)^{1/2}$. 
For parameters $r_\perp, \rpar>0$ such that $r_\perp \ll \rpar \ll 1$, define the rescaled cut-off functions
\begin{equation*}
\theta_{r_\perp}(x_1,x_2)
=
\frac{1}{r_\perp}
\theta\left(\frac{x_1}{r_\perp},\frac{x_2}{r_\perp}\right),
\quad
\Theta_{r_\perp}(x_1,x_2)
=
\frac{1}{r_\perp}
\Theta\left(\frac{x_1}{r_\perp},\frac{x_2}{r_\perp}\right),
\quad 
\psi_{\rpar}(x_3)
=
\frac{1}{\rpar^{1/2}}
\psi\left(\frac{x_3}{\rpar}\right).
\end{equation*}
We periodize $\theta_{r_\perp}, \Theta_{r_\perp}$ and $\psi_{\rpar}$ so that they are viewed as periodic functions on $\mathbb{T}^2, \mathbb{T}^2$ and $\mathbb{T}$ respectively.
Consider a large real number $\lambda$ such that $\lambda r_\perp\in\mathbb{N}$, and a large time oscillation parameter $\mu>0$. For every $\xi\in \Lambda$ we introduce
\begin{equation*}\aligned
\psi_{\xi}(x,t)
&:=
\psi_{\rpar}(n_*r_\perp\lambda(x\cdot \xi+\mu t)),
\\ 
\Theta_{\xi}(x)
&:=
\Theta_{r_{\perp}}(n_*r_\perp\lambda(x-\alpha_\xi)\cdot A_\xi, n_*r_\perp\lambda(x-\alpha_\xi)\cdot(\xi\times A_\xi)),
\\
\theta_{\xi}(x)
&:=
\theta_{r_{\perp}}(n_*r_\perp\lambda(x-\alpha_\xi)\cdot A_\xi, n_*r_\perp\lambda(x-\alpha_\xi)\cdot(\xi\times A_\xi)),
\endaligned
\end{equation*}
where $x \in \T^3$, $t \in \R$ and $\alpha_\xi\in\mathbb{R}^3$ are shifts to ensure that $\{\Theta_{\xi}\}_{\xi\in\Lambda}$ have mutually disjoint support.

The intermittent jets $W_{\xi}:\mathbb{T}^3\times \mathbb{R}\rightarrow\mathbb{R}^3$ are then defined for every $\xi \in \Lambda$ as
\begin{equation*}
W_{\xi}(x,t)
:=
\xi\,\psi_{\xi}(x,t) \, \theta_{\xi}(x).
\end{equation*}
Each $W_{\xi}$ has zero space average and is $(\T/r_\perp \lambda)^3$-periodic in its space variable.
Moreover, by the choice of $\alpha_\xi$ we have that
$$
W_{\xi}\otimes W_{\xi'}\equiv0, 
\quad \textrm{for } \xi\neq \xi'\in\Lambda,
$$
and for every $t \in \R$
$$
\frac{1}{(2\pi)^3}\int_{\mathbb{T}^3}
W_{\xi}(x,t)\otimes W_{\xi}(x,t)dx=\xi\otimes\xi.
$$
These facts, combined with \cite[Lemma 6.6]{BuVi19b}, imply that
\begin{equation}\label{geometric equality}
\frac{1}{(2\pi)^3}
\sum_{\xi\in\Lambda}\gamma_\xi^2(R)
\int_{\mathbb{T}^3}W_{\xi}(x,t)\otimes W_{\xi}(x,t)dx=R,
\quad
\forall R \in \overline{B_{1/2}}(\mathrm{Id}),
\quad
\forall t \in \R.
\end{equation}

Since $W_{\xi}$ is not divergence free, we introduce the compressibility corrector term
\begin{equation*}
W_{\xi}^{(c)}
:=
\frac{1}{n_*^2\lambda^2} \nabla \psi_{\xi}\times \textrm{curl}(\Theta_{\xi}\xi)
=
\curl\curl V_{\xi}-W_{\xi},
\quad
V_{\xi}
:=
\frac{1}{n_*^2\lambda^2}\,\xi\,\psi_\xi\,\Theta_\xi,
\end{equation*}
so that 
$$
\dvg \left(W_{\xi}+W_{\xi}^{(c)}\right) = 0.
$$

Next, we  recall the key   bounds from \cite[Section 7.4]{BuVi19b}. For integers $0 \leq N, M\leq 10$ and $p\in [1,\infty]$ the following hold provided $\rpar^{-1}\ll r_{\perp}^{-1}\ll \lambda$
\begin{align*}
\|\nabla^N \partial_t^M \psi_{\xi}\|_{C_t L^p_x}
&\lesssim 
\rpar^{1/p-1/2} \left(\frac{r_\perp\lambda}{\rpar}\right)^N
\left(\frac{r_\perp\lambda \mu}{\rpar}\right)^M,
\\
\|\nabla^N \theta_{\xi}\|_{ L^p_x}+\|\nabla^N\Theta_{\xi}\|_{ L^p_x}
&\lesssim 
r^{2/p-1}_\perp\lambda^N,
\\
\|\nabla^N \partial_t^M W_{\xi}\|_{C_t L^p_x}
&\lesssim 
r^{2/p-1}_\perp \rpar^{1/p-1/2}\lambda^N\left(\frac{r_\perp\lambda\mu}{\rpar}\right)^M,
\\
\|\nabla^N \partial_t^MW_{\xi}^{(c)}\|_{C_t L^p_x}
&\lesssim 
r_\perp^{2/p} \rpar^{1/p-3/2}\lambda^N\left(\frac{r_\perp\lambda\mu}{\rpar}\right)^M,
\\
\|\nabla^N \partial_t^M V_{\xi}
\|_{C_t L^p_x}
&\lesssim 
r^{2/p-1}_\perp \rpar^{1/p-1/2}\lambda^{N-2}\left(\frac{r_\perp\lambda\mu}{\rpar}\right)^M,
\end{align*}
where $\lesssim$ is an abbreviation for inequality $\leq$ up to an unimportant multiplicative constant, that we keep implicit for notational simplicity. 
In the lines above, the implicit constants may depend on $p$, but are independent of $N,M$ (because we restricted to $N,M \leq 10$) and $\lambda, r_\perp, \rpar, \mu$.

As a notational convention, hereafter we denote with the bold symbols $\boldsymbol{\psi}_\xi$, $\boldsymbol{\theta}_\xi$ etc. the respective quantities precomposed with the flow $\phi_{n+1}$, namely $\boldsymbol{\psi}_\xi := \psi_\xi \circ \phi_{n+1}$, $\boldsymbol{\theta}_\xi := \theta_\xi \circ \phi_{n+1}$ etc.
When composing with the flow $\phi_{n+1}$ these bounds become
\begin{align*}
\|\nabla^N \partial_t^M \boldsymbol\psi_{\xi}\|_{C_t L^p_x}
&\lesssim 
\rpar^{1/p-1/2} \left(\frac{r_\perp\lambda}{\rpar}\right)^N
\left(\frac{r_\perp\lambda \mu}{\varsigma_{n+1}\rpar}\right)^M,
\\
\|\nabla^N \partial_t^M \boldsymbol\theta_{\xi}\|_{C_t L^p_x}+\|\nabla^N \partial_t^M \boldsymbol\Theta_{\xi}\|_{C_t L^p_x}
&\lesssim 
r^{2/p-1}_\perp\lambda^N \left(\frac{\lambda}{\varsigma_{n+1}} \right)^M,
\\
\|\nabla^N \partial_t^M \mathbf{W}_{\xi}\|_{C_t L^p_x}
&\lesssim 
r^{2/p-1}_\perp \rpar^{1/p-1/2}\lambda^N\left(\frac{r_\perp\lambda\mu}{\varsigma_{n+1}\rpar}\right)^M,
\\
\|\nabla^N \partial_t^M \mathbf{W}_{\xi}^{(c)}\|_{C_t L^p_x}
&\lesssim 
r_\perp^{2/p} \rpar^{1/p-3/2}\lambda^N\left(\frac{r_\perp\lambda\mu}{\varsigma_{n+1}\rpar}\right)^M,
\\
\|\nabla^N \partial_t^M \mathbf{V}_{\xi}
\|_{C_t L^p_x}
&\lesssim 
r^{2/p-1}_\perp \rpar^{1/p-1/2}\lambda^{N-2}\left(\frac{r_\perp\lambda\mu}{\varsigma_{n+1}\rpar}\right)^M.
\end{align*}

\section{Construction of solutions with prescribed energy} \label{sec:convex_energy}

As usual in convex integration schemes, for every $n \in \N$ we will consider the Navier-Stokes-Reynolds system for $t \in \R$
\begin{align} \label{eq:random_NS-reynolds}
\partial_t v_n 
&+ 
\dvgphien(v_n\otimes v_n)
+
\nablaphien q_n
-
\Dphien v_n
=
\dvgphien \mathring{R}_n,
\end{align}
where the \emph{Reynolds stress} $\mathring{R}_n$ takes values in the space of $3 \times 3$ symmetric traceless matrices.
Similarly to what done in \cite{HoLaPa22+} we do not impose the divergence-free condition $\dvg^{\phi_n} v_n = 0$ at every $n \in \N$, but rather we only require a decay of certain norms of $\dvg^{\phi_n} v_n$ along the iteration.
For technical reasons we also require $\int_{\T^3} v_n = 0$ for every $n \in \N$.
Then, given a sequence $(v_n,q_n,\phi_n,\mathring{R}_n)$ of progressively measurable solutions to the random Navier-Stokes-Reynolds system \eqref{eq:random_NS-reynolds}, we construct a solution $v$ of \eqref{eq:random_NS} showing the convergences, with respect to suitable topologies:
\begin{align*}
v_n \to v,
\quad
\mathring{R}_n \to 0,
\quad
\dvg^{\phi_n} v_n \to 0.
\end{align*} 
Notice that $\phi_n \to \phi$ by construction, and we do not need to prove the convergence of $q_n$ since the limit $q$ can be recovered by $v$, $\phi$, and the condition $\int_{\T^3} q = 0$.

\subsection{Iterative assumptions and main proposition}
We shall define
\begin{align*}
\lambda_n := a^{b^n},
\qquad
\delta_n := \lambda_1^{2\beta} \lambda_n^{-2\beta},
\end{align*} 
for some parameters $a \in 2\N$, $b \in 7\N$ sufficiently large and $\beta \in (0,1)$ sufficiently small.

Without loss of generality we can suppose that the energy profile $e$ in the statement of \autoref{thm:energy} is defined for negative times also, and $\inf_{t \in \R} e(t) > \underline{e} > 0$ and $\| e \|_{C^1_t} < \overline{e} < \infty$ (possibly changing the values of $\underline{e}$ and $\overline{e}$).

We require the following iterative estimates on the energy error
\begin{equation} \label{est:energy}
\frac{3}{4} \delta_{n+1} e(t) 
\leq  
e(t) - \int_{\mathbb{T}^3} |v_n(x,t)|^2 dx 
\leq 
\frac{5}{4} \delta_{n+1} e(t),
\quad
\forall t \leq \mathfrak{t},
\end{equation}
as well as the integrability bounds
\begin{align} \label{est:integrab}
\| q_n \|_{C_{\mathfrak{t}} L^1_x}
\leq 
C_q \lambda_n^{1/1000},
\qquad
\| \mathring{R}_n \|_{C_{\mathfrak{t}} L^1_x}
\leq 
C_R  \delta_{n+2},
\end{align}
the derivative controls
\begin{gather} \label{est:deriv}
\| v_n \|_{C_{\mathfrak{t}}W_x^{1,1}}
\leq 
C_v \sum_{k=0}^n \delta_{k+2}^2,
\qquad
\| v_n \|_{C^1_{\mathfrak{t},x}}
\leq 
C_v \lambda_n^{6},
\qquad
\| v_n \|_{C^1_{\mathfrak{t}}C_x^3}
\leq 
C_v \lambda_n^{12},
\\
\| q_n \|_{C_{\mathfrak{t}} C_x^2}
\leq 
C_q \lambda_n^{12},
\qquad
\| \mathring{R}_n \|_{C_{\mathfrak{t}} C_x^1}
\leq 
C_R \lambda_n^{20}, 
\end{gather}
and the estimates on the divergence
\begin{equation} \label{est:div}
\| \dvgphien v_n \|_{C_{\mathfrak{t}} H_x^{-1}}
\leq 
C_v \delta_{n+3}^3,
\qquad
\| \dvgphien v_n \|_{C_{\mathfrak{t}} L^1_x}
\leq 
C_v \delta_{n+3}^3.
\end{equation}
for some universal constants $C_v,C_q \in (0,\infty)$ and $C_R \in (0,\overline{e}/48)$.

\autoref{thm:energy} is then a consequence of the following:
\begin{prop}\label{prop:it}
There exists a choice of parameters $a,b,\beta$ with the following property.
Let $(v_n,q_n,\phi_n,\mathring{R}_n)$, $n \in \N$ be a solution of \eqref{eq:random_NS} with $\phi_n$ given by \eqref{eq:def_phi_n} and satisfying the inductive estimates \eqref{est:energy}, \eqref{est:integrab}, \eqref{est:deriv} and \eqref{est:div}.
Then there exists a quadruple $(v_{n+1},q_{n+1},\phi_{n+1},\mathring{R}_{n+1})$, solution of \eqref{eq:random_NS} with $\phi_{n+1}$ given by \eqref{eq:def_phi_n}, satisfying the same inductive estimates with $n$ replaced by $n+1$ and such that
\begin{align} \label{eq:difference}
\| v_{n+1} - v_{n} \|_{C_{\mathfrak{t}}L_x^2}
\leq
C_v \delta_{n+1}^{1/2},
\qquad
\| v_{n+1} - v_{n} \|_{C_{\mathfrak{t}}W_x^{1,1}}
\leq
C_v \delta_{n+3}^2.
\end{align}
Moreover, $(v_{n+1},q_{n+1},\mathring{R}_{n+1})$ at any time $t$ depends only on the values of $(v_n,q_n,\mathring{R}_n,\phi_{n+1},e)$ at times $s \leq t$.
\end{prop}

\begin{proof}[Proof of \autoref{thm:energy}]
Starting the iteration from $(v_0,q_0,\mathring{R}_0)=(0,0,0)$ it is easy to check that $\{v_n\}_{n \in \N}$ is a Cauchy sequence in $C_\mathfrak{t} H_x^\gamma \cap C_\mathfrak{t} W_x^{1+\gamma,1}$. Moreover, $\mathring{R}_n$ and $\dvgphien v_n$ converge to zero in $C_\mathfrak{t} L^1_x$ and thus the limit $v$ is indeed a solution of \eqref{eq:random_NS} on $[0,\mathfrak{t}]$. Then \eqref{est:energy} guarantees the solution has the desired energy profile.
Moreover, $v$ is progressively measurable as the limit of progressive processes.
Convergences on an arbitrary time interval $[0,\mathfrak{t}_L]$, $L>1$ has been previously discussed, and are omitted. We just mention the fact that on $[0,\mathfrak{t}_L]$ the iterative estimates gain a factor $L_n := L^{m^n}$ for some $m$ (see also \cite[Proposition 2.4]{HoLaPa22+}) and all the convergences (except for the energy profile) hold true as soon as $m<b$. Since $b$ will be taken large in what follows, simply taking $m:=b-1$ gives the desired result.  
Finally, if $e_1$, $e_2$ are two energy profiles and $e_1(t)=e_2(t)$ for every $t\in[0,T/2]$, then the two associated solutions coincide up to time $T/2$ by the last part of the proposition. 
\end{proof}
The remainder of this section is devoted to the proof of \autoref{prop:it}.

\subsection{Perturbation of the velocity}
\label{ssec:velocity_est}
The building blocks of the perturbation are the intermittent jets presented in \autoref{ssec:jets}.
First of all, we fix a choice for the parameters:
\begin{align*}
r_\perp := \lambda_{n+1}^{-6/7},
\quad
\rpar := \lambda_{n+1}^{-4/7},
\quad
\mu := \lambda_{n+1}^{9/7},
\end{align*}
as well as the mollification parameters
\begin{align*}
\ell := \lambda_{n+1}^{-\frac{3}{2} \alpha} \lambda_n^{-100},
\qquad
\varsigma_n = \frac{\lambda_n^{-1/100}}{(n+1)^4}.
\end{align*}
In the expression above $\alpha \in (0,1)$ is a sufficiently small parameter to be chosen later. We require that $\ell^{-1}$ is an integer power of $2$, which is true for instance if $\alpha b \in 2\N$.

Next, let $\chi_1 \in C^\infty_c([-1,1]^3 \times [0,1))$ be a standard mollifier, and denote $\chi_\ell$ the rescaled kernel $\chi_\ell(x,t) := \ell^{-4}\chi_1(x/\ell,t/\ell)$, and define
\begin{align*}
v_\ell \coloneqq v_n \ast \chi_\ell,
\quad
q_\ell \coloneqq q_n \ast \chi_\ell,
\quad
\mathring{R}_\ell \coloneqq \mathring{R}_n \ast \chi_\ell.
\end{align*}

The new velocity field $v_{n+1}$ will be constructed as a perturbation of $v_\ell$:
\begin{align*}
v_{n+1}
:=
v_\ell + w_{n+1},
\end{align*}
where $w_{n+1} := w^{(p)}_{n+1} + w^{(c)}_{n+1} + w^{(t)}_{n+1}$ is split in three different contributions described below.

\subsubsection{The principal perturbation $w^{(p)}_{n+1}$}

Following \cite{HoZhZh22b+}, we define the energy pumping term\footnote{For large times $t > \mathfrak{t}$ the energy pumping term $\gamma_n(t)$ must be defined differently, along the lines of \cite[Remark 3.2]{HoLaPa22+}. This is due to the fact that $|v_n(x,t)|^2$ may be too large for large values of $t$, resulting in possibly vanishing of $\rho$. This is ultimately the reason why we can impose the desired energy profile to solutions only up to time $\mathfrak{t}$.}
\begin{align*}
\gamma_n (t)
&:=
\frac{1}{3(2\pi)^3} \left( e(t)(1-\delta_{n+2}) - \int_{\T^3} |v_n (x,t)|^2 dx \right),
\quad
t \in (-\infty,\mathfrak{t} ],
\end{align*}
and denote $\gamma_\ell := \gamma_n \ast \chi_\ell$ and $\rho := 2 \sqrt{\ell^2 + |\mathring{R}_\ell|^2} + \gamma_\ell$.
For every $\xi \in \Lambda$ let us introduce the amplitude functions
\begin{align} \label{eq:amplitude}
\mathbf{a}_\xi 
:=
\rho^{1/2} \gamma_{\xi} \left( \mathrm{Id} - \frac{\mathring{R}_\ell}{\rho} \right),
\qquad
a_\xi 
:=
\mathbf{a}_\xi \circ \phi_{n+1}^{-1},
\end{align}
and the principal part of the perturbation
\begin{align} \label{eq:def_wp}
w^{(p)}_{n+1} 
:=
\sum_{\xi \in \Lambda} \mathbf{a}_{\xi} \mathbf{W}_{\xi}
:= 
\left( \sum_{\xi \in \Lambda} a_{\xi} W_{\xi} \right) \circ \phi_{n+1},
\end{align}
where we use the bold symbols to indicate precomposition with the flow $\phi_{n+1}$. 
By \eqref{geometric equality} and the fact that $\phi_{n+1}$ is measure preserving it holds
\begin{align} \label{eq:wp otimes wp}
w^{(p)}_{n+1} \otimes w^{(p)}_{n+1} + \mathring{R}_\ell - \rho Id
&=
\sum_{\xi \in \Lambda} \mathbf{a}_{\xi}^2\, \Pi_{\neq 0} (\mathbf{W}_{\xi} \otimes \mathbf{W}_{\xi}) 
\\ \nonumber
&=
\left( \sum_{\xi \in \Lambda} a_{\xi}^2\, \Pi_{\neq 0} (W_{\xi} \otimes W_{\xi}) \right) \circ \phi_{n+1}.
\end{align}

\subsubsection{The compressibility corrector $w^{(c)}_{n+1}$} \label{ssec:w_c}

Denote $\mathcal{P}$ the classical Leray projector on zero-average, divergence free velocity fields and $\mathcal{Q} = Id - \mathcal{P}$ its orthogonal.
Recall from \cite{HoLaPa22+} the operators $\mathcal{P}^{\phi_n}, \mathcal{Q}^{\phi_n}$, $n \in \N$, acting on a given $v \in C^\infty(\T^3,\R^3)$ as
\begin{align*}
\mathcal{P}^{\phi_n} v \coloneqq
[\mathcal{P}(v \circ \phi_n^{-1})] \circ \phi_n,
\quad
\mathcal{Q}^{\phi_n} v \coloneqq
[\mathcal{Q}(v \circ \phi_n^{-1})] \circ \phi_n.
\end{align*}

The compressibility corrector is made of two different contributions, namely
\begin{align*}
w^{(c)}_{n+1} := w^{(c,1)}_{n+1}+w^{(c,2)}_{n+1},
\end{align*}
where
\begin{align} \label{eq:def_wc1}
w^{(c,1)}_{n+1} 
\coloneqq 
-(\mathcal{Q}^{\phi_n}v_n) \ast \chi_\ell,
\end{align}
is as in \cite{HoLaPa22+} and is needed to reduce the size of $\dvgphin v_\ell$, whereas
\begin{align} \label{eq:def_wc2}
w^{(c,2)}_{n+1}
&:=
\sum_{\xi \in \Lambda}
\curlphin(\nablaphin \mathbf{a}_{\xi} \times \mathbf{V}_{\xi}) + \nablaphin \mathbf{a}_{\xi} \times \curlphin \mathbf{V}_{\xi} + \mathbf{a}_{\xi} \mathbf{W}^{(c)}_{\xi}
\\
&:= \nonumber
\left( \sum_{\xi \in \Lambda}
\curl(\nabla a_{\xi} \times V_{\xi}) + \nabla a_{\xi} \times \curl V_{\xi} + a_{\xi} W^{(c)}_{\xi} \right) \circ \phi_{n+1},
\end{align}
serves to compensate for the divergence of the principal corrector since
\begin{align*}
w^{(p)}_{n+1}+w^{(c,2)}_{n+1}
&=
\sum_{\xi \in \Lambda} \curlphin \curlphin (\mathbf{a}_{\xi} \mathbf{V}_{\xi} )
\\
&=
\left( \sum_{\xi \in \Lambda} \curl \curl (a_{\xi} V_{\xi} )  \right) \circ \phi_{n+1}
\end{align*}
and thus $\dvgphin (w^{(p)}_{n+1}+w^{(c,2)}_{n+1}) = 0$.

\subsubsection{The temporal corrector $w^{(t)}$}
Finally, the temporal corrector $w^{(t)}_{n+1}$ is defined as 
\begin{align} \label{eq:def_wt}
w^{(t)}_{n+1}
&:=
- \frac{1}{\mu} \mathcal{P}^{\phi_{n+1}}
\Pi_{\neq 0}
\sum_{\xi \in \Lambda} 
\mathbf{a}_{\xi}^2\, \boldsymbol{\theta}_{\xi}^2\, \boldsymbol{\psi}_{\xi}^2\, \xi 
\\
&:= \nonumber
- \frac{1}{\mu} \left( \mathcal{P}
\Pi_{\neq 0}
\sum_{\xi \in \Lambda} 
a_{\xi}^2 \theta_{\xi}^2 \psi_{\xi}^2 \xi \right) \circ \phi_{n+1}.
\end{align}

By definition it holds
\begin{align} \label{eq:dt_wt}
\partial_t w^{(t)}_{n+1}
&=
- \frac{1}{\mu} \left( \mathcal{P}
\Pi_{\neq 0}
\sum_{\xi \in \Lambda} 
\partial_t  \left(a_{\xi}^2 \theta_{\xi}^2 \psi_{\xi}^2 \xi \right) \right) \circ \phi_{n+1}
\\ \nonumber
&\quad
- \frac{1}{\mu} \dot\phi_{n+1} \cdot \left[ \nabla 
\left( \mathcal{P}
\Pi_{\neq 0}
\sum_{\xi \in \Lambda} 
a_{\xi}^2 \theta_{\xi}^2 \psi_{\xi}^2 \xi  \right) \circ \phi_{n+1} \right]
\\ \nonumber
&=
- \frac{1}{\mu} \mathcal{P}^{\phi_{n+1}}
\left(\Pi_{\neq 0} \left(
\sum_{\xi \in \Lambda}  
\partial_t  \left(a_{\xi}^2 \theta_{\xi}^2 \psi_{\xi}^2 \xi \right) \right) \circ \phi_{n+1} \right)
\\ \nonumber
&\quad
- \frac{1}{\mu} \dot\phi_{n+1} \cdot \left[ \nabla 
\left( \mathcal{P}
\Pi_{\neq 0}
\sum_{\xi \in \Lambda} 
a_{\xi}^2 \theta_{\xi}^2 \psi_{\xi}^2 \xi  \right) \circ \phi_{n+1} \right]
\\ \nonumber
&=
- \frac{1}{\mu} \mathcal{P}^{\phi_{n+1}}
\Pi_{\neq 0} 
\sum_{\xi \in \Lambda}  
\partial_t  \left(\mathbf{a}_{\xi}^2 \boldsymbol\theta_{\xi}^2 \boldsymbol\psi_{\xi}^2 \xi \right)
\\ \nonumber
&\quad
+ \frac{1}{\mu} \mathcal{P}^{\phi_{n+1}}
\Pi_{\neq 0} \left( \dot\phi_{n+1} \cdot \left[ \sum_{\xi \in \Lambda}  
\nabla \left(a_{\xi}^2 \theta_{\xi}^2 \psi_{\xi}^2 \xi \right) \circ \phi_{n+1}
\right] \right)
\\ \nonumber
&\quad
- \frac{1}{\mu} \dot\phi_{n+1} \cdot \left[ \nabla 
\left( \mathcal{P}
\Pi_{\neq 0}
\sum_{\xi \in \Lambda} 
a_{\xi}^2 \theta_{\xi}^2 \psi_{\xi}^2 \xi  \right) \circ \phi_{n+1} \right].
\end{align}

\subsubsection{Estimates on the velocity}
The iterative estimates \eqref{est:deriv} on the velocity $v_{n+1}$ and \eqref{eq:difference} on the increment $w_{n+1} = v_{n+1}-v_n$ (actually, with $W_x^{1,1}$ replaced by $W_x^{1,p}$ for some $p>1$ sufficiently close to one) are easily obtained from the corresponding bounds in \cite[Section 3]{HoZhZh22b+} and the fact that the flow $\phi_{n+1}$ is measure preserving.

First, by (3.23), (3.24) and (3.28) in \cite{HoZhZh22b+} it holds $\| \rho \|_{C_{\mathfrak{t}}L^1_x} \leq C \delta_{n+1}$ for some universal constant $C$, as well as
\begin{align*}
\| a_\xi \|_{C_{\mathfrak{t},x}}
&=
\| \mathbf{a}_\xi \|_{C_{\mathfrak{t},x}}
\lesssim 
\| \rho \|_{C_{\mathfrak{t},x}}^{1/2} 
\lesssim 
\ell^{-2} \delta_{n+1}^{1/2},
\\
\| a_\xi \|_{C_{\mathfrak{t}}L_x^2}
&=
\| \mathbf{a}_\xi \|_{C_{\mathfrak{t}}L_x^2}
\lesssim 
\| \rho \|_{C_{\mathfrak{t}}L^1_x}^{1/2} 
\lesssim 
\delta_{n+1}^{1/2},
\end{align*}
whereas by (3.25) and (3.34) therein we have for every $N = 1,2,...10$ and $M=0,1$
\begin{align*}
\| \rho \|_{C^N_{\mathfrak{t},x}} 
&\lesssim 
\ell^{2-7N} \delta_{n+1},
\\
\| a_\xi \|_{C^M_{\mathfrak{t}}C^N_x} 
&\lesssim
\varsigma_{n+1}^{-M}
\| \mathbf{a}_\xi \|_{C^{M+N}_{\mathfrak{t},x}}
\lesssim
\varsigma_{n+1}^{-M} \ell^{-8-7(M+N)} \delta_{n+1}^{1/2}.
\end{align*}

By the previous inequalities, together with \cite[Lemma 7.4]{BuVi19b}, the fact that $\phi_{n+1}$ is measure preserving, and $W_\xi = \mathbf{W}_\xi \circ \phi_{n+1}^{-1}$ is $1/r_\perp \lambda$ periodic in its space variable, one has
\begin{align*}
\| w^{(p)}_{n+1} \|_{C_{\mathfrak{t}}L_x^2}
=
\| w^{(p)}_{n+1} \circ \phi_{n+1}^{-1}\|_{C_{\mathfrak{t}}L_x^2}
&\leq \frac{C_v}{2}
\delta_{n+1}^{1/2}
,
\end{align*}
where the constant $C_v$ is universal.

For $p \in (1,\infty)$, the bounds (3.43), (3.44), (3.45) read as
\begin{align*}
\| w^{(p)}_{n+1} \|_{C_{\mathfrak{t}}L_x^p}
&\lesssim
\delta_{n+1}^{1/2}
\ell^{-8} r_\perp^{2/p-1} \rpar^{1/p-1/2},
\\
\| w^{(c,2)}_{n+1} \|_{C_{\mathfrak{t}}L_x^p}
&\lesssim
\delta_{n+1}^{1/2}
\ell^{-22} r_\perp^{2/p} \rpar^{1/p-3/2},
\\
\| w^{(t)}_{n+1} \|_{C_{\mathfrak{t}}L_x^p}
&\lesssim
\delta_{n+1}
\ell^{-16} r_\perp^{2/p-1} \rpar^{1/p-2} \lambda_{n+1}^{-1},
\end{align*}
Moreover, arguing as in \cite[Lemma 4.9]{HoLaPa22+} and invoking iterative assumption \eqref{est:div} we get
\begin{align*}
\| w^{(c,1)}_{n+1} \|_{C_{\mathfrak{t}}L_x^2}
&\lesssim
\| \dvgphien v_n \|_{C_{\mathfrak{t}}H_x^{-1}}
\lesssim
\delta_{n+3}^3.
\end{align*}
In particular, all the bounds above with $p=2$ guarantee that the iterative assumption on $\| v_{n+1} - v_n \|_{C_\mathfrak{t}L_x^2}$ holds true, up to choosing the parameter $a$ large enough so to absorb all the implicit constants in the previous inequalities.

We also need to bound the $W_x^{1,p}$ norm of the increment for some $p>1$. It holds (see (3.52) and (3.51) in \cite{HoZhZh22b+})
\begin{align*}
\| w^{(p)}_{n+1} + w^{(c,2)}_{n+1} \|_{C_{\mathfrak{t}}W_x^{1,p}}
&\lesssim
\sum_{\xi \in \Lambda}
\| \curlphin \curlphin (\mathbf{a}_\xi \mathbf{V}_\xi) \|_{C_{\mathfrak{t}}W_x^{1,p}}
\\
&\lesssim
r_\perp^{2/p-1} \rpar^{1/p-1/2} \ell^{-8} \lambda_{n+1},
\\
\| w^{(t)}_{n+1} \|_{C_{\mathfrak{t}}W_x^{1,p}}
&\lesssim
\frac{1}{\mu} \sum_{\xi \in \Lambda}
\| \mathbf{a}_\xi^2 \,\boldsymbol\theta_\xi^2 \,\boldsymbol\psi_\xi^2 \|_{C_{\mathfrak{t}}W_x^{1,p}}
\\
&\lesssim
r_\perp^{2/p-2} \rpar^{1/p-1} \ell^{-16} \lambda_{n+1}^{-2/7}.
\end{align*}
In addition, by interpolation and assuming $p \in (1,2)$ it holds
\begin{align*}
\| w^{(c,1)}_{n+1} \|_{C_{\mathfrak{t}}W_x^{1,p}}
&\lesssim
\| \dvgphien v_n \|_{C_{\mathfrak{t}}L_x^p}
\\
&\lesssim
\| \dvgphien v_n \|_{C_{\mathfrak{t}}L^1_x}^{2-p}
\| v_n \|_{C^1_{\mathfrak{t},x}}^{p-1}
\\
&\lesssim
\delta_{n+3}^{3(2-p)} \lambda_{n}^{6(p-1)},
\end{align*}
and thus $\|v_{n+1}-v_n\|_{C_{\mathfrak{t}}W_x^{1,p}} \lesssim \delta_{n+3}^2$ at least when $\beta$ is sufficiently small, $p < \min \{ 101/100, 1+ \beta b^3 /100 \}$, and $a$ is taken large enough. The iterative bound \eqref{est:deriv} on $\|v_{n+1}\|_{C_{\mathfrak{t}}W_x^{1,1}}$ descends immediately.

Finally, for the $C^1_{\mathfrak{t},x}$ and $C^1_{\mathfrak{t}}C_x^3$ norms of the incremental velocity we have the very loose estimates (cf. (3.48), (3.49) and (3.50) in \cite{HoZhZh22b+})
\begin{align*}
\| w^{(p)}_{n+1} + w^{(c,2)}_{n+1} \|_{C^1_{\mathfrak{t},x}}
&\lesssim
\sum_{\xi \in \Lambda}
\| \curlphin \curlphin (\mathbf{a}_\xi \mathbf{V}_\xi) \|_{C^1_{\mathfrak{t},x}}
\\
&\lesssim
\sum_{\xi \in \Lambda}
\varsigma_{n+1}^{-2} \| \mathbf{a}_\xi \|_{C^1_{\mathfrak{t}}C^2_x} 
(\| \mathbf{V}_\xi \|_{C^1_{\mathfrak{t}}C_x^2} 
+ 
\| \mathbf{V}_\xi \|_{C_{\mathfrak{t}}C_x^3})
\\
&\lesssim
\sum_{\xi \in \Lambda}
\varsigma_{n+1}^{-3} \ell^{-29} 
(\| \mathbf{V}_\xi \|_{C^1_{\mathfrak{t}}W_x^{4,2}} 
+ 
\| \mathbf{V}_\xi \|_{C_{\mathfrak{t}}W_x^{5,2}})
\\
&\lesssim
\varsigma_{n+1}^{-4} \ell^{-29} r_\perp \rpar^{-1}\lambda_{n+1}^3 \mu
\lesssim \lambda_{n+1}^5,
\\
\| w^{(p)}_{n+1} + w^{(c,2)}_{n+1} \|_{C^1_{\mathfrak{t}}C_x^3}
&\lesssim
\varsigma_{n+1}^{-3} \lambda_{n+1}^2
\sum_{\xi \in \Lambda}
\| \curlphin \curlphin (\mathbf{a}_\xi \mathbf{V}_\xi) \|_{C_{\mathfrak{t}}C_x^3}
\\
&\lesssim
\varsigma_{n+1}^{-3} \lambda_{n+1}^2
\sum_{\xi \in \Lambda}
\| \mathbf{a}_\xi \|_{C_{\mathfrak{t}}C^5_x} \| \mathbf{V}_\xi \|_{C_{\mathfrak{t}}C_x^5}
\\
&\lesssim
\varsigma_{n+1}^{-3} \lambda_{n+1}^2
\sum_{\xi \in \Lambda}  \ell^{-43} \| \mathbf{V}_\xi \|_{C_{\mathfrak{t}}W_x^{7,2}}
\\
&\lesssim
\varsigma_{n+1}^{-3} \ell^{-43} \lambda_{n+1}^7 
\lesssim \lambda_{n+1}^{8},
\end{align*}
and
\begin{align*}
\| w^{(t)}_{n+1} \|_{C^1_{\mathfrak{t},x}}
&\lesssim
\frac{\lambda_{n+1}^2}{\varsigma_{n+1}\mu} \sum_{\xi \in \Lambda}
\| \mathbf{a}_\xi^2 \,\boldsymbol\theta_\xi^2 \,\boldsymbol\psi_\xi^2 \|_{C_{\mathfrak{t},x}}
+
\frac{1}{\mu} \sum_{\xi \in \Lambda}
\| \mathbf{a}_\xi^2 \,\boldsymbol\theta_\xi^2 \,\boldsymbol\psi_\xi^2 \|_{C_{\mathfrak{t}}C_x^1}
\\
&\lesssim
\frac{\lambda_{n+1}^2}{\varsigma_{n+1}\mu} \sum_{\xi \in \Lambda}
\| \mathbf{a}_\xi\|_{C_{\mathfrak{t},x}}^2 
\| \boldsymbol\theta_\xi^2 \,\boldsymbol\psi_\xi^2 \|_{C_{\mathfrak{t},W_x^{2,2}}}
+
\frac{1}{\mu} \sum_{\xi \in \Lambda}
\| \mathbf{a}_\xi\|_{C_{\mathfrak{t}}C_x^1}^2 
\|\boldsymbol\theta_\xi^2 \,\boldsymbol\psi_\xi^2 \|_{C_{\mathfrak{t}}W_x^{3,2}}
\\
&\lesssim
\varsigma_{n+1}^{-1} \mu^{-1} \ell^{-30} r_\perp^{-1} \rpar^{-1/2}\lambda_{n+1}^6
\lesssim
\lambda_{n+1}^6,
\\
\| w^{(t)}_{n+1} \|_{C^1_{\mathfrak{t}}C_x^3}
&\lesssim
\frac{\lambda_{n+1}^2}{\varsigma_{n+1}\mu} \sum_{\xi \in \Lambda}
\| \mathbf{a}_\xi^2 \,\boldsymbol\theta_\xi^2 \,\boldsymbol\psi_\xi^2 \|_{C_{\mathfrak{t}}C_x^3}
\\
&\lesssim
\frac{\lambda_{n+1}^2}{\varsigma_{n+1}\mu} \sum_{\xi \in \Lambda}
\| \mathbf{a}_\xi^2 \|_{C_{\mathfrak{t}}C_x^3}
\|\boldsymbol\theta_\xi^2 \,\boldsymbol\psi_\xi^2 \|_{C_{\mathfrak{t}}W_x^{5,2}}
\\
&\lesssim
\varsigma_{n+1}^{-1}\mu^{-1}
\ell^{-58}
r_\perp^{-1} \rpar^{-1/2} \lambda_{n+1}^{12}
\lesssim
\lambda_{n+1}^{12}.
\end{align*}

For $w^{(c,1)}_{n+1}$ we have instead
\begin{align*}
\|w^{(c,1)}_{n+1}\|_{C^3_{\mathfrak{t},x}}
&\lesssim
\ell^{-5} 
\|\mathcal{Q}^{\phi_n}v_n\|_{C_{\mathfrak{t}}L_x^2}
\lesssim
\ell^{-5} 
\|\dvgphien v_n \|_{C_{\mathfrak{t}}H_x^{-1}}
\lesssim
\ell^{-5} \delta_{n+3}^3.
\end{align*}
Notice that these estimates guarantee that \eqref{est:deriv} are satisfied provided $\beta$ (resp. $a$) is taken once more sufficiently small (resp. large).

\subsubsection{Estimates on the divergence}
To verify the iterative bounds \eqref{est:div} it is necessary to reformulate Lemma 4.7 of \cite{HoLaPa22+} so to estimate the difference $(\dvgphin - \dvgphien)\,v$ in the $L_x^p$ and $L_x^2$ scales. The product map $(f,g) \mapsto fg$ is continuous from $L_x^p \times C_x \to L_x^p$ and $H_x^{-1} \times C_x^2 \to H_x^{-1}$, therefore we have:
\begin{lem} \label{lem:G_holder} 
For every $p \in [1,\infty]$ there exists a constant $C$ with the following property. For every $n \in \N$, given any smooth vector field $v \in C^\infty(\mathbb{T}^3,\R^3)$ on the torus and denoting $G \coloneqq \left(\dvgphin-\dvgphien\right) v$, almost surely it holds for every $t \leq \mathfrak{t}$
\begin{align*}
\|G(t)\|_{H_x^{-1}}
&\leq 
C (n+1)\varsigma_n^{1/4}
\| v \|_{L_x^2},
\\
\|G(t)\|_{L_x^p}
&\leq 
C (n+1)\varsigma_n^{1/4}
\| v \|_{W_x^{1,p}}.
\end{align*}
\end{lem}

Then, arguing as in \cite{HoLaPa22+}, rewrite
\begin{align} \label{eq:divergence} 
\dvgphin v_{n+1}
&= 
\dvgphin v_\ell - \left( \dvgphin v_n \right) \ast \chi_\ell
\\ 
&\quad+ \nonumber
\left( \dvgphin v_n \right) \ast \chi_\ell
-
\left(\dvgphien v_n\right)   \ast \chi_\ell
\\
&\quad+ \nonumber
\left(\dvgphien Q^{\phi_n}v_n\right) \ast \chi_\ell
-
\dvgphien \left( (Q^{\phi_n}v_n) \ast \chi_\ell \right)
\\
&\quad+  \nonumber
\dvgphien \left( (Q^{\phi_n}v_n) \ast \chi_\ell \right)
-
\dvgphin \left( (Q^{\phi_n}v_n) \ast \chi_\ell \right).
\end{align}

By \cite[Lemma 4.4]{HoLaPa22+} and iterative assumptions  it holds
\begin{align*}
\left\|\dvgphin v_\ell - \left( \dvgphin v_n \right) \ast \chi_\ell \right\|_{C_{\mathfrak{t},x}}
&\lesssim
\ell^{1/4} \|v_n\|_{C^1_{\mathfrak{t},x}}
\lesssim
\lambda_{n}^{-10},
\\
\left\|
\left(\dvgphien Q^{\phi_n}v_n\right) \ast \chi_\ell
-
\dvgphien \left( (Q^{\phi_n}v_n) \ast \chi_\ell \right) \right\|_{C_{\mathfrak{t},x}}
&\lesssim
\ell^{1/4} \|Q^{\phi_n} v_n\|_{C_{\mathfrak{t}}C^1_x}
\\
&\lesssim
\ell^{1/4} \| v_n\|_{C^1_{\mathfrak{t},x}}
\lesssim
\lambda_{n}^{-10}.
\end{align*}
A fortiori, these quantities satisfy the same bounds in 
$C_{\mathfrak{t}}H_x^{-1}$ and $C_{\mathfrak{t}}L_x^p$.

The other two terms on the right-hand-side of \eqref{eq:divergence} are controlled with \autoref{lem:G_holder} and the bounds on $\|v_n\|_{C_{\mathfrak{t}}L_x^2},\|v_n\|_{C_{\mathfrak{t}}W_x^{1,p}}$, and using $(n+1) \varsigma_{n}^{1/4} = \lambda_n^{-1/400} \leq \frac{C_v}{4} \delta_{n+3}^3$ for $\beta$ sufficiently small.

\subsubsection{Estimate on the energy}
The control on the energy is quite standard.
Rewrite
\begin{align*}
|v_{n+1}|^2
=
|v_\ell|^2
+
|w^{(p)}_{n+1}|^2
+
|w^{(c)}_{n+1}+w^{(t)}_{n+1}|^2
+
2 v_\ell \cdot w^{(p)}_{n+1}
+
2 (v_\ell + w^{(p)}_{n+1})\cdot(w^{(c)}_{n+1}+w^{(t)}_{n+1}).
\end{align*}

Since $\| w^{(c)}_{n+1}+w^{(t)}_{n+1} \|_{C_{\mathfrak{t}}L_x^2} \lesssim \delta_{n+3}^2$ it holds
\begin{align*}
\left| \int_{\T^3} (v_\ell + w^{(p)}_{n+1})\cdot(w^{(c)}_{n+1}+w^{(t)}_{n+1}) dx \right|
&\lesssim 
\| v_\ell + w^{(p)}_{n+1} \|_{C_{\mathfrak{t}}L_x^2}
\| w^{(c)}_{n+1}+w^{(t)}_{n+1} \|_{C_{\mathfrak{t}}L_x^2}
\lesssim
\delta_{n+3}^2.
\end{align*}

The term $v_\ell \cdot w^{(p)}_{n+1}$ is controlled in the following way:
\begin{align*}
\int_{\T^3} |v_\ell \cdot w^{(p)}_{n+1}| dx
&\lesssim
\| v_\ell \|_{C_t L_x^\infty} \|w^{(p)}_{n+1}\|_{C_tL^1_x}
\lesssim
\ell^{-10} r_\perp \rpar^{1/2}
\lesssim
\delta_{n+3}^2. 
\end{align*}

By \eqref{eq:wp otimes wp} and using that $\mathrm{tr}(\mathring{R}_\ell)=0$ and $\phi_{n+1}$ is measure preserving, we have
\begin{align*}
\int_{\T^3}
\left( |w^{(p)}_{n+1}|^2 
- 3\rho \right) dx
&=
\sum_{\xi \in \Lambda}
\int_{\T^3}
\mathbf{a}_\xi^2 \,\Pi_{\neq 0}|\mathbf{W}_\xi|^2 dx
=
\sum_{\xi \in \Lambda}
\int_{\T^3}
{a}_\xi^2 \,\Pi_{\neq 0}|{W}_\xi|^2 dx
\\
&=
\sum_{\xi \in \Lambda}
\int_{\T^3}
{a}_\xi^2 \,\Pi_{\geq r_\perp \lambda_{n+1}/2}|{W}_\xi|^2 dx
\\
&\lesssim
\| a_\xi^2 \|_{C_\mathfrak{t}C_x^9} (\lambda_{n+1}r_\perp)^{-9} \| |W_\xi|^2 \|_{C_\mathfrak{t}L_x^2}
\\
&\lesssim
\ell^{7-16\cdot 9} (\lambda_{n+1}r_\perp)^{-9} r_\perp^{-1} \rpar^{-1/2}
\lesssim
\lambda_{n+1}^{-1/10}. 
\end{align*}
The last line is justified by the fact that ${W}_\xi$ is $r_\perp \lambda_{n+1}$ periodic.
On the other hand,
\begin{align*}
\int_{\T^3} 3\rho dx 
=
\gamma_\ell - \gamma_n
+
6 \int_{\T^3} \sqrt{\ell^2 + |\mathring{R}_\ell|^2} dx
+
e(t)(1-\delta_{n+2}) - \int_{\T^3} |v_n(x,t)|^2 dx.
\end{align*}
Putting all together, we can rewrite the energy error at level $n+1$ as
\begin{align*}
e(t)(1-\delta_{n+2}) - \int_{\T^3} |v_{n+1}|^2 dx
&=
\gamma_n-\gamma_\ell
+
6\int_{\T^3} \sqrt{\ell^2 + |\mathring{R}_\ell|^2} dx
+
\int_{\T^3} |v_{\ell}|^2 - |v_{n}|^2 dx
+
r_{n+1},
\end{align*}
where by the previous lines $r_{n+1}$ is a reminder satisfying $|r_{n+1}| \leq \frac{\overline{e}}{8} \delta_{n+2}$ (up to choosing $a$ sufficiently large).
Then \eqref{est:energy} is recovered noticing that all the other terms on the right-hand-side of the equation above are smaller than $\overline{e}\ell + 4C_v^2\lambda_n^{6} \ell + 6 C_R \delta_{n+2} \leq \frac{\overline{e}}{8} \delta_{n+2}$.

\subsection{The oscillatory term and the new pressure $q_{n+1}$}
\label{ssec:oscillation}
The temporal corrector serves to reduce the oscillatory term
\begin{align*}
\partial_t w^{(t)}_{n+1}  
+
\Pi_{\neq 0}\sum_{\xi \in \Lambda} 
\mathbf{a}_{\xi}^2 \dvgphin (\mathbf{W}_{\xi} \otimes \mathbf{W}_{\xi}).
\end{align*}

For classical intermittent jets the key identity was 
\begin{align*}
\dvg(W_{\xi} \otimes W_{\xi})
&=
2 (W_{\xi} \cdot \nabla \psi_{\xi}) \theta_{\xi} \xi
=
\frac{1}{\mu} \theta_{\xi}^2 \partial_t \psi_{\xi}^2 \xi
=
\frac{1}{\mu}\partial_t  \left( \theta_{\xi}^2 \psi_{\xi}^2 \xi\right),
\end{align*}
which holds true because $\xi \cdot \nabla \psi = \mu^{-1} \partial_t \psi$.
However, for us it holds instead
\begin{align} \label{eq:div_WxW}
\dvgphin(\mathbf{W}_{\xi} \otimes \mathbf{W}_{\xi}) 
&=
\dvg(W_{\xi} \otimes W_{\xi}) \circ \, \phi_{n+1}
\\
&= \nonumber
\frac{1}{\mu} \partial_t  \left( \theta_{\xi}^2 \psi_{\xi}^2 \xi\right) \circ \phi_{n+1}
\\
&= \nonumber
\frac{1}{\mu} \partial_t  \left( \boldsymbol{\theta}_{\xi}^2 \boldsymbol{\psi}_{\xi}^2 \xi \right)
-
\frac{1}{\mu} \dot{\phi}_{n+1} \cdot \left[ \nabla \left( \theta_{\xi}^2 \psi_{\xi}^2 \xi\right) \circ \phi_{n+1} \right].
\end{align}

Therefore by \eqref{eq:dt_wt} and \eqref{eq:div_WxW} we have 
\begin{align} \label{eq:oscillation_error_decomp}
\partial_t &w^{(t)}_{n+1}  
+
\Pi_{\neq 0}\sum_{\xi \in \Lambda} 
\mathbf{a}_{\xi}^2 \, \dvgphin (\mathbf{W}_{\xi} \otimes \mathbf{W}_{\xi})
\\
&= \nonumber
\frac{1}{\mu} \mathcal{Q}^{\phi_{n+1}}
\Pi_{\neq 0} 
\sum_{\xi \in \Lambda}  
\partial_t  \left(\mathbf{a}_{\xi}^2 \boldsymbol\theta_{\xi}^2 \boldsymbol\psi_{\xi}^2 \xi \right)
\\
&\quad \nonumber
+ \frac{1}{\mu} \mathcal{P}^{\phi_{n+1}}
\Pi_{\neq 0} \,\dot\phi_{n+1} \cdot \left[ \left(
\sum_{\xi \in \Lambda}  
\nabla \left(a_{\xi}^2 \theta_{\xi}^2 \psi_{\xi}^2 \xi \right) \right) \circ \phi_{n+1}
\right]
\\
&\quad \nonumber
- \frac{1}{\mu} \dot\phi_{n+1} \cdot \left[ \nabla 
\left( \mathcal{P}
\Pi_{\neq 0}
\sum_{\xi \in \Lambda} 
a_{\xi}^2 \theta_{\xi}^2 \psi_{\xi}^2 \xi  \right) \circ \phi_{n+1} \right]
\\
&\quad   \nonumber
-
\frac{1}{\mu}\Pi_{\neq 0}
\sum_{\xi \in \Lambda} 
(\partial_t  \mathbf{a}_{\xi}^2) \boldsymbol\theta_{\xi}^2 \boldsymbol\psi_{\xi}^2 \xi
\\
&\quad \nonumber
-
\frac{1}{\mu} \Pi_{\neq 0}
\sum_{\xi \in \Lambda} 
\mathbf{a}_{\xi}^2\, 
\dot{\phi}_{n+1} \cdot \left[\nabla \left( \theta_{\xi}^2 \psi_{\xi}^2 \xi\right) \circ \phi_{n+1} \right].
\end{align}

The lines second to five on the right-hand-side will be shown to be small in $L_x^p$ in \autoref{ssec:R&P}. The key fact is that in these terms there is no time derivative acting on $\psi_\xi$ or $\boldsymbol{\psi}_\xi$, and thus the factor $\mu^{-1}$ in front dominates by our choice of parameters.

On the other hand, the first term on the right-hand-side is the gradient of a pressure (in the sense that equals $\nabla^{\phi_{n+1}} \tilde{P}$ for some $\tilde{P}$) but it needs to be manipulated further, nonetheless. 
Indeed, we can not prove our iterative estimate on the $\| q_n \|_{C_\mathfrak{t}L^1_x}$ by estimating the increment $\tilde{P}$ in $C_\mathfrak{t}L^1_x$.

We rewrite
\begin{align} \label{eq:oscillation_error_decomp_II}
\frac{1}{\mu} \mathcal{Q}^{\phi_{n+1}}
\Pi_{\neq 0} 
\sum_{\xi \in \Lambda}  
\partial_t  \left(\mathbf{a}_{\xi}^2 \boldsymbol\theta_{\xi}^2 \boldsymbol\psi_{\xi}^2 \xi \right)
&=
\frac{1}{\mu} \mathcal{Q}
\left(
\Pi_{\neq 0} 
\sum_{\xi \in \Lambda}  
[\partial_t  \left(\mathbf{a}_{\xi}^2 \boldsymbol\theta_{\xi}^2 \boldsymbol\psi_{\xi}^2 \xi \right)] \circ \phi_{n+1}^{-1}
\right) \circ \phi_{n+1}
\\
&= \nonumber
\frac{1}{\mu} \mathcal{Q}
\left(
\Pi_{\neq 0} 
\sum_{\xi \in \Lambda}  
\partial_t  \left({a}_{\xi}^2 \theta_{\xi}^2 \psi_{\xi}^2 \xi \right)
\right) \circ \phi_{n+1}
\\
&\quad+ \nonumber
\frac{1}{\mu} \mathcal{Q}
\left(
\Pi_{\neq 0}  (\dot \phi_{n+1} \circ \phi_{n+1}^{-1}) \cdot
 \left[ \sum_{\xi \in \Lambda} \nabla \left( 
{a}_{\xi}^2 \theta_{\xi}^2 \psi_{\xi}^2 \xi \right) \right]
\right) \circ \phi_{n+1}
\\
&=: \nonumber
(\nabla P) \circ \phi_{n+1}
+
\frac{1}{\mu} \mathcal{Q}^{\phi_{n+1}}
\Pi_{\neq 0} \left(\dot\phi_{n+1} \cdot 
\left[ \sum_{\xi \in \Lambda}\nabla \left( 
{a}_{\xi}^2 \theta_{\xi}^2 \psi_{\xi}^2 \xi \right)\circ \phi_{n+1} \right] \right)
\\
&= \nonumber
\nablaphin (P \circ \phi_{n+1})
+
\frac{1}{\mu} \mathcal{Q}^{\phi_{n+1}}
\Pi_{\neq 0} \left( 
\dot\phi_{n+1} \cdot  
\left[ \sum_{\xi \in \Lambda} \nabla \left( 
{a}_{\xi}^2 \theta_{\xi}^2 \psi_{\xi}^2 \xi \right)\circ \phi_{n+1} \right] \right),
\end{align}
where the pressure increment $P$ is implicitly defined in the second-to-last line. The other term does not have any time derivative and therefore can be easily absorbed into the Reynolds stress $\mathring{R}_{n+1}$, in particular it does not need to be incorporated in the new pressure $q_{n+1}$.

As a consequence, recalling the bounds $\|\rho\|_{C^2_{\mathfrak{t},x}} \lesssim \ell^{-12}$ and $\|\rho\|_{C_{\mathfrak{t}}L^1_x} \lesssim \delta_{n+1}$ and defining the new pressure as
\begin{align} \label{eq:new_pressure}
q_{n+1}
:=
q_\ell 
-
\rho
-
P \circ \phi_{n+1},
\end{align}
the iterative estimate \eqref{est:deriv} on the derivative of the pressure is satisfied since
\begin{align*}
\| P \circ \phi_{n+1} \|_{C_\mathfrak{t}C_x^2}
&\lesssim
\left\| \frac{1}{\mu} \Pi_{\neq 0} 
\sum_{\xi \in \Lambda}  
\partial_t  ({a}_{\xi}^2 \theta_{\xi}^2 \psi_{\xi}^2 \xi) \right\|_{C_\mathfrak{t}C_x^2}
\\
&\lesssim 
\frac{\lambda_{n+1}^2}{\varsigma_{n+1}\mu}
\sum_{\xi \in \Lambda}  
\left\| {a}_{\xi}^2 \theta_{\xi}^2 \psi_{\xi}^2  \right\|_{C_\mathfrak{t}C_x^2}
\\
&\lesssim 
\frac{\lambda_{n+1}^2}{\varsigma_{n+1}\mu}
\sum_{\xi \in \Lambda}  
\| {a}_{\xi}\|_{C_\mathfrak{t}C_x^2}^2 
\| \theta_{\xi}^2 \psi_{\xi}^2  \|_{C_\mathfrak{t}W_x^{4,2}}
\\
&\lesssim 
\varsigma_{n+1}^{-1} \mu^{-1} \ell^{-44} r_\perp^{-1} \rpar^{-1/2} \lambda_{n+1}^{10}
\lesssim 
\lambda_{n+1}^{10},
\end{align*}
whereas the integrability bound \eqref{est:integrab} is satisfied as long as we can prove $\| P \circ \phi_{n+1}\|_{L^1_x} = \|P\|_{L^1_x} \leq \lambda_{n+1}^{1/1000}$ for every $n \in \N$, assuming at least $C_q \geq 2$ and $\beta$ sufficiently small.
We postpone the verification of the latter inequality to \autoref{ssec:pressure_est}.

\subsection{The Reynolds stress $\mathring{R}_{n+1}$} \label{ssec:R&P}

Let us recall from \cite{DLS13} the operator $\mathcal{R}$ that acts as left inverse of the operator $\mbox{div}$.
Namely, for every $v \in C^\infty(\T^3,\R^3)$ let $\mathcal{R}v$ be the matrix-valued function defined in \cite[Definition 4.2]{DLS13}, so that $\mathcal{R}v$ takes values in the space of symmetric trace-free matrices and $\dvg\mathcal{R}v = v - \frac{1}{(2\pi)^3}\int_{\T^3}v$.
Then we have
\begin{align*}
\mathcal{R}^{\phi_{n+1}} v \coloneqq
[\mathcal{R}(v \circ \phi_{n+1}^{-1})] \circ \phi_{n+1},
\quad
\dvgphin (\mathcal{R}^{\phi_{n+1}} v) = v - \frac{1}{(2\pi)^3}\int_{\T^3}v.
\end{align*} 

Then, we shall choose the new Reynolds stress $\mathring{R}_{n+1}$ such that 
\begin{align} \label{eq:R&q}
\mathring{R}_{n+1} 
\coloneqq
\mathcal{R}^{\phi_{n+1}}
\left( \partial_t v_{n+1} + \dvgphin(v_{n+1} \otimes v_{n+1}) + \nablaphin q_{n+1} - \Dphin v_{n+1}\right).
\end{align}
It is easy to check that the term inside the parentheses has zero space average by construction, and thus \eqref{eq:R&q} gives a solution to the Navier-Stokes-Reynolds system at level $n+1$.
In addition, it can be conveniently decomposed as 
\begin{align*}
\partial_t v_{n+1} 
&+ 
\dvgphin(v_{n+1}\otimes v_{n+1})
+
\nablaphin q_{n+1}
-
\Dphin v_{n+1}
\\
&= \nonumber
\underbrace{\left[\partial_t w^{(p)}_{n+1} + \partial_t w^{(c)}_{n+1} + \dvgphin(w_{n+1} \otimes v_\ell + v_\ell \otimes w_{n+1} ) - \Dphin w_{n+1} \right]}_{=\,linear\, error}
\\
&\quad+ \nonumber
\underbrace{\left[\dvgphin\left(w^{(p)}_{n+1} \otimes w^{(p)}_{n+1} +\mathring{R}_\ell\right) + \partial_t w^{(t)}_{n+1} + \nablaphin(q_{n+1}-q_\ell)\right]}_{=\,oscillation\, error}
\\
&\quad+ \nonumber
\underbrace{\left[ \dvgphin(v_\ell \otimes v_\ell - (v_n \otimes v_n) \ast \chi_\ell)\right]}_{=\,mollification\, error\,I}
\\
&\quad+ \nonumber
\underbrace{\left[ \dvgphien\left( \left( v_n \otimes v_n\right) \ast \chi_\ell  + q_\ell Id - \mathring{R}_\ell  \right)
-
\left( \dvgphien \left( v_n \otimes v_n + q_n Id - \mathring{R}_n \right) \right) \ast \chi_\ell
 \right]}_{=\,mollification\, error\, II}
 \\
&\quad+ \nonumber
\underbrace{\left[ \chi_\ell\ast\Dphien v_n - \Dphien v_\ell \right]}_{=\,mollification\, error \,III}
\\
&\quad+ \nonumber
\underbrace{\left[(\Dphien - \Dphin)\, v_\ell +\left(\dvgphin-\dvgphien\right)
((v_n \otimes v_n) \ast \chi_\ell - \mathring{R}_\ell + q_\ell Id) \right]}_{=\,flow\, error}
\\
&\quad+ \nonumber
\underbrace{\left[\dvgphin( w^{(p)}_{n+1} \otimes (w^{(c)}_{n+1}+w^{(t)}_{n+1}) + (w^{(c)}_{n+1}+w^{(t)}_{n+1}) \otimes w_{n+1}) \right]}_{=\,corrector\, error}.
\end{align*}
In this way, since the operator $\mathcal{R}^{\phi_{n+1}}$ is linear, we are able to separately control the different contributions to the new Reynolds stress $\mathring{R}_{n+1}$.
Thus, the estimate \eqref{est:integrab} on the Reynolds stress in $L^1_x$ (actually in $L_x^p$ for some $p>1$) is obtained in a standard way, making use of the antidivergence $\mathcal{R}^{\phi_{n+1}}$.
The ``new" error terms coming from the composition with the flow are easily controlled assuming $\varsigma_{n+1}^{-1} \ll \lambda_{n+1} \ll \mu$. 
The key is that in those errors we don't have time derivatives of the intermittent jets.

\subsubsection{Linear error}
The terms
\begin{align*}
\mathcal{R}^{\phi_{n+1}} \dvgphin(w_{n+1} \otimes v_\ell + v_\ell \otimes w_{n+1})
\end{align*}
and
\begin{align*}
\mathcal{R}^{\phi_{n+1}} \Dphin w_{n+1}
=
\mathcal{R}^{\phi_{n+1}} \dvgphin \nablaphin w_{n+1}
\end{align*}
are controlled using that $\mathcal{R}^{\phi_{n+1}} \dvgphin :L_x^p \to L_x^p$ is bounded, Young convolution inequality and
\begin{align*}
\|v_\ell \otimes w_{n+1}\|_{L_x^p} + \|\nablaphin w_{n+1}\|_{L_x^p} 
&\lesssim 
\|v_\ell\|_{L_x^\infty}\|w_{n+1}-w^{(c,1)}_{n+1}\|_{L_x^p} 
\\
&+
\|v_\ell\|_{L_x^{2p'}}\|w^{(c,1)}_{n+1}\|_{L_x^{2}}
+
\|w_{n+1}\|_{W_x^{1,p}}
\\
&\lesssim 
\|v_\ell\|_{L_x^\infty}\|w_{n+1}-w^{(c,1)}_{n+1}\|_{L_x^p} 
\\
&+
\ell^{-4 \frac{p'-1}{p'}}
\|v_n\|_{L_x^2}\|w^{(c,1)}_{n+1}\|_{L_x^2} 
+
\|w_{n+1}\|_{W_x^{1,p}}
\lesssim
\delta_{n+3}^2.
\end{align*}
Here $p'$ is such that $1/2p' + 1/2 = 1/p$, in particular $p' \to 1^+$ as $p \to 1^+$, so that the factor $\ell^{-4 \frac{p'-1}{p'}}\delta_{n+3} \lesssim 1$ for $p$ sufficiently small (depending only on $\alpha$, $\beta$ and $b$ but not on $a$; in particular we can always increase the value of $a$ to absorb any implicit constant). 

As for the other term, recall
\begin{align*}
w^{(p)}_{n+1} + w^{(c,2)}_{n+1}
&= 
\curlphin \curlphin \mathbf{V}
=
\left( \curl \curl V \right) \circ \phi_{n+1},
\\
\partial_t w^{(p)}_{n+1} + \partial_t w^{(c,2)}_{n+1}
&=
\left( \curl \curl \partial_t V \right) \circ \phi_{n+1}
\\
&\quad+
\dot{\phi}_{n+1} \cdot \left[ 
\left(\nabla  \curl \curl V \right) \circ \phi_{n+1}
\right]
\\
&=
\curlphin \left( (\curl \partial_t V)  \circ \phi_{n+1}\right)
\\
&\quad+
\dot{\phi}_{n+1} \cdot \left[ 
\left(\nabla  \curl \curl V \right) \circ \phi_{n+1}
\right],
\end{align*}
and use that $\mathcal{R}^{\phi_{n+1}}\curlphin$ and $\mathcal{R}^{\phi_{n+1}}$ are bounded on $L_x^p$ plus the estimates (take $p$ close to one)
\begin{align*}
\| (\curl \partial_t V)  \circ \phi_{n+1} \|_{L_x^p}
&\lesssim
\| V \|_{C^1_{\mathfrak{t}} W_x^{1,p}}
\\
&\lesssim
r_\perp^{2/p} \rpar^{1/p-3/2} \mu 
\lesssim
\lambda_{n+1}^{-1/10},
\\
\| \dot{\phi}_{n+1} \cdot \left[ 
\left(\nabla  \curl \curl V \right) \circ \phi_{n+1}
\right] \|_{L_x^p}
&\lesssim
\varsigma_{n+1}^{-1} 
\| V \|_{C_{\mathfrak{t}} W_x^{3,p}}
\\
&\lesssim
\varsigma_{n+1}^{-1}
\lambda_{n+1} r_\perp^{2/p-1} \rpar^{1/p-1/2}
\lesssim
\lambda_{n+1}^{-1/10}.
\end{align*}

It only remains to control $\mathcal{R}^{\phi_{n+1}} \partial_t w^{(c,1)}_{n+1}$, which can be done as in \cite{HoLaPa22+}.
More precisely, one writes $\partial_t w^{(c,1)}_{n+1}$ as $- (\mathcal{Q}^{\phi_n} v_n) \ast \partial_t \chi^0_\ell$ (we use the zero-mean version of $\partial_t \chi_\ell$) and bounds (using Lemma C.7 and adapting Lemma C.6 therein with any $\delta>0$, $1+1/p=1/p_1 + 1/p_2$ and $p_2$ sufficiently close to $1$)
\begin{align*}
\| \mathcal{R}^{\phi_{n+1}} \partial_t w^{(c,1)}_{n+1} \|_{C_{\mathfrak{t}}L_x^p}
&\lesssim
\|(\mathcal{Q}^{\phi_n} v_n) \ast \partial_t \chi^0_\ell\|_{C_{\mathfrak{t}}W_x^{-1,p}}
\\
&\lesssim
\ell
\sup_{s \leq t}
\|(\mathcal{Q}^{\phi_n} v_n)(\cdot,t-s) \ast_{\T^3} \partial_t \chi^0_\ell (\cdot,s)\|_{C_{\mathfrak{t}}W_x^{-1,p}}
\\
&\lesssim
\ell
\sup_{s \leq t}
\|(\mathcal{Q}^{\phi_n} v_n)(\cdot,t-s) \ast_{\T^3} \partial_t \chi^0_\ell (\cdot,s)\|_{C_{\mathfrak{t}}B^{\delta-1}_{p,\infty}}
\\
&\lesssim
\ell
\sup_{s \leq t}
\|\mathcal{Q}^{\phi_n} v_n\|_{C_{\mathfrak{t}}L_x^{p_1}}
\| \partial_t \chi^0_\ell \|_{C_{\mathfrak{t}}B^{2\delta-1}_{p_2,\infty}}
\\
&\lesssim
\ell^{-3\delta} \delta_{n+3}^3
\lesssim 
\delta_{n+3}^2.
\end{align*}
In the lines above $B^{\alpha}_{p,q} := B^{\alpha}_{p,q}(\T^3)$, $\alpha \in \R$ and $p,q \in [1,\infty]$, denotes the Besov space on the three dimensional torus, cf. \cite{HoLaPa22+}.

\subsubsection{Oscillation error}
Recalling \eqref{eq:wp otimes wp}, it holds
\begin{align*}
\partial_t w^{(t)}_{n+1}  
&+
\dvgphin \left(
w^{(p)}_{n+1} \otimes w^{(p)}_{n+1} + \mathring{R}_\ell - \rho Id
\right)
\\
&=
\partial_t w^{(t)}_{n+1}
+
\dvgphin 
\left(
\sum_{\xi \in \Lambda} \mathbf{a}_{\xi}^2\, \Pi_{\neq 0} (\mathbf{W}_{\xi} \otimes \mathbf{W}_{\xi}) 
\right)
\\
&=
\partial_t w^{(t)}_{n+1}
+
\Pi_{\neq 0}\sum_{\xi \in \Lambda} 
\mathbf{a}_{\xi}^2 \dvgphin (\mathbf{W}_{\xi} \otimes \mathbf{W}_{\xi})
\\
&\quad+
\Pi_{\neq 0} 
\sum_{\xi \in \Lambda} 
\nablaphin \mathbf{a}_{\xi}^2 \cdot \Pi_{\neq 0}(\mathbf{W}_{\xi} \otimes \mathbf{W}_{\xi}).
\end{align*}

According to the same estimates as for $R_{osc}^{(x)}$ in \cite[page 22]{HoZhZh22b+} we can bound for $p$ sufficiently small (recall that $\mathbf{a}_\xi$ and $a_\xi$ enjoy the same bounds on space derivatives up to unimportant multiplicative constants)
\begin{align*}
\left\| \mathcal{R}^{\phi_{n+1}} \left( 
\Pi_{\neq 0} 
\sum_{\xi \in \Lambda} 
\nablaphin \mathbf{a}_{\xi}^2 \cdot \Pi_{\neq 0}(\mathbf{W}_{\xi} \otimes \mathbf{W}_{\xi})
\right) \right\|_{C_\mathfrak{t}L_x^p}
&=
\left\| \mathcal{R}\left( 
\sum_{\xi \in \Lambda} 
\nabla a_{\xi}^2  \cdot \Pi_{\neq 0} (W_{\xi} \otimes W_{\xi})
\right) \right\|_{C_\mathfrak{t}L_x^p}
\\
&=
\left\| \mathcal{R}\left( 
\sum_{\xi \in \Lambda} 
\nabla a_{\xi}^2  \cdot \Pi_{\geq r_\perp \lambda_{n+1}/2} (W_{\xi} \otimes W_{\xi})
\right) \right\|_{C_\mathfrak{t}L_x^p}
\\
&\lesssim
\ell^{-23} r_\perp^{2/p-3} \rpar^{1/p-1} \lambda_{n+1}^{-1}
\lesssim \lambda_{n+1}^{-1/10}.
\end{align*}

As for the other terms, we have already seen by \eqref{eq:oscillation_error_decomp} and \eqref{eq:oscillation_error_decomp_II} the decomposition 
\begin{align*}
\partial_t &w^{(t)}_{n+1}  
+
\Pi_{\neq 0}\sum_{\xi \in \Lambda} 
\mathbf{a}_{\xi}^2 \, \dvgphin (\mathbf{W}_{\xi} \otimes \mathbf{W}_{\xi})
\\
&=
\nablaphin (P \circ \phi_{n+1})
\\
&\quad+
\frac{1}{\mu} 
\Pi_{\neq 0} \, \dot\phi_{n+1} \cdot
 \left[ \sum_{\xi \in \Lambda} \nabla \left( 
{a}_{\xi}^2 \theta_{\xi}^2 \psi_{\xi}^2 \xi \right)\circ \phi_{n+1} \right]
\\
&\quad
- \frac{1}{\mu} \dot\phi_{n+1} \cdot \left[ \nabla 
\left( \mathcal{P}
\Pi_{\neq 0}
\sum_{\xi \in \Lambda} 
a_{\xi}^2 \theta_{\xi}^2 \psi_{\xi}^2 \xi  \right) \circ \phi_{n+1} \right]
\\
&\quad  
-
\frac{1}{\mu}\Pi_{\neq 0}
\sum_{\xi \in \Lambda} 
(\partial_t  \mathbf{a}_{\xi}^2) \boldsymbol\theta_{\xi}^2 \boldsymbol\psi_{\xi}^2 \xi
\\
&\quad
-
\frac{1}{\mu} \Pi_{\neq 0}
\sum_{\xi \in \Lambda} 
\mathbf{a}_{\xi}^2\, 
\dot{\phi}_{n+1} \cdot \left[\nabla \left( \theta_{\xi}^2 \psi_{\xi}^2 \xi\right) \circ \phi_{n+1} \right],
\end{align*}
where we can bound
\begin{align*}
\frac{1}{\mu} \left\| 
\Pi_{\neq 0}
\sum_{\xi \in \Lambda} 
(\partial_t  \mathbf{a}_{\xi}^2) \boldsymbol\theta_{\xi}^2 \boldsymbol\psi_{\xi}^2 \xi
\right\|_{C_\mathfrak{t}L_x^p}
&\lesssim
\frac{1}{\mu} \sum_{\xi \in \Lambda} 
\| a_{\xi} \|_{C_{\mathfrak{t},x}} 
\| a_{\xi} \|_{C^1_{\mathfrak{t},x}} 
\| \theta_{\xi}^2 \psi_{\xi}^2 \|_{C_\mathfrak{t}L_x^p}
\\
&\lesssim
\mu^{-1} \varsigma_{n+1}^{-1} \ell^{-19} \rpar^{1/p-1} r_\perp^{2/p-2}
\lesssim
\lambda_{n+1}^{-1},
\end{align*}
and all the other terms except $\nablaphin (P \circ \phi_{n+1})$ with
\begin{align*}
\frac{\varsigma_{n+1}^{-1}}{\mu}
\sum_{\xi \in \Lambda}
\| a_{\xi}^2 \theta_{\xi}^2 \psi_{\xi}^2 \|_{C_\mathfrak{t}W_x^{1,p}}
&\lesssim
\frac{\varsigma_{n+1}^{-1}}{\mu}
\| a_{\xi} \|_{C_{t,x}} \| a_{\xi} \|_{C^1_{\mathfrak{t},x}} \| \theta_{\xi} \|_{L_x^{2p}}^2 \| \psi_{\xi} \|_{L_x^{2p}}^2
\\
&\quad+
\frac{\varsigma_{n+1}^{-1}}{\mu}
\| a_{\xi} \|_{C_{t,x}}^2 \| \theta_{\xi} \|_{L_x^{2p}} \| \nabla \theta_{\xi} \|_{L_x^{2p}} \|\psi_{\xi}\|_{L_x^{2p}}^2
\\
&\quad+
\frac{\varsigma_{n+1}^{-1}}{\mu}
\| a_{\xi} \|_{C_{t,x}}^2 \| \theta_{\xi} \|_{L_x^{2p}}^2 \| \nabla \psi_{\xi} \|_{L_x^{2p}} \|\psi_{\xi}\|_{L_x^{2p}}
\\
&\lesssim
\varsigma_{n+1}^{-1} \mu^{-1}
\ell^{-16} r_\perp^{2/p-2} \rpar^{1/p-1}
( \ell^{-7} + \lambda_{n+1} + r_\perp \rpar^{-1} \lambda_{n+1} )
\lesssim
\lambda_{n+1}^{-1/10}.
\end{align*}

\subsubsection{Mollification error}
In order to control the mollification error, use 
\begin{align*}
\| \mathcal{R}^{\phi_{n+1}} \dvgphin
(v_\ell \otimes v_\ell - (v_n \otimes v_n) \ast \chi_\ell)
\|_{C_{\mathfrak{t}}L_x^p}
&\lesssim
\|v_\ell \otimes v_\ell - (v_n \otimes v_n) \ast \chi_\ell
\|_{C_{\mathfrak{t}}L_x^p}
\\
&\lesssim
\ell \lambda_{n}^{6}
\lesssim
\delta_{n+3}^2,
\end{align*}
and by \cite[Lemma 4.4]{HoLaPa22+} with $G =  v_n \otimes v_n + q_n Id - \mathring{R}_n$
\begin{align*}
\left\| \dvgphien\left( G \ast \chi_\ell \right)
-
\left( \dvgphien G \right) \ast \chi_\ell
\right\|_{C_{\mathfrak{t},x}}
\lesssim
\ell^{1/4} \|G\|_{C_{\mathfrak{t}}C_x^1}
\lesssim
\ell^{1/4} \lambda_n^{20}
\lesssim
\delta_{n+3}^2.
\end{align*}

As for the term $\chi_\ell \ast \Dphien v_n - \Dphien v_\ell$, rewrite
\begin{align*}
\chi_\ell \ast \Dphien v_n - \Dphien v_\ell
&=
\chi_\ell \ast \dvgphien \nablaphien v_n 
-
\dvgphien (\chi_\ell \ast \nablaphien v_n )
\\
&\quad+
\dvgphien (\chi_\ell \ast \nablaphien v_n )
-  \dvgphien \nablaphien v_\ell,
\end{align*}
and use the same lemma with $G=\nablaphien v_n$ (or $G=v_n$ and replacing $\dvgphien$ with $\nablaphien$)
\begin{align*}
\left\|
\chi_\ell \ast \dvgphien \nablaphien v_n 
-
\dvgphien (\chi_\ell \ast \nablaphien v_n )
\right\|_{C_{\mathfrak{t},x}}
&\lesssim
\ell^{1/4} \|v_n\|_{C_{\mathfrak{t}}C^2_x}
\\
&\lesssim
\ell^{1/4} \lambda_n^{12} 
\lesssim
\delta_{n+3}^2,
\\
\left\| \mathcal{R}^{\phi_{n+1}} \left(
\dvgphien (\chi_\ell \ast \nablaphien v_n )
-  \dvgphien \nablaphien v_\ell \right)
\right\|_{C_{\mathfrak{t}}L_x^p}
&\lesssim
\| \chi_\ell \ast \nablaphien v_n - \nablaphien v_\ell
\|_{C_{\mathfrak{t}}L_x^p}
\\
&\lesssim
\ell^{1/4} \lambda_n^{12}
\lesssim
\delta_{n+3}^2.
\end{align*}

\subsubsection{Flow error}

Use \autoref{lem:G_holder} with $G = \nablaphien v_\ell$ or $G=(v_n \otimes v_n) \ast \chi_\ell - \mathring{R}_\ell + q_\ell Id$ (or again with $G=v_\ell$ and replacing $(\dvgphin-\dvgphien)$ with $(\nablaphin-\nablaphien)$ ) to get
\begin{align*}
\left\| (\Dphien- \Dphin) v_\ell \right\|_{C_{\mathfrak{t}}W_x^{-1,p}}
&\leq
\left\| (\dvgphien- \dvg^{\phi_{n+1}}) \nablaphien v_\ell \right\|_{C_{\mathfrak{t}}W_x^{-1,p}}
\\
&+
\left\| \dvg^{\phi_{n+1}}(\nablaphien- \nablaphin) v_\ell \right\|_{C_{\mathfrak{t}}W_x^{-1,p}}
\\
&\lesssim
(n+1) \varsigma_n^{1/4} \|v_\ell\|_{C_{\mathfrak{t}}W_x^{1,p}}
\lesssim
\delta_{n+3}^{2},
\end{align*}
and for some $\delta>0$ depending only on $p$ and such that $\delta \to 0^+$ as $p\to 1^+$ 
\begin{align*}
&\left\| 
\left(\dvgphin-\dvgphien\right)
((v_n \otimes v_n) \ast \chi_\ell - \mathring{R}_\ell + q_\ell Id)
\right\|_{C_{\mathfrak{t}}W_x^{-1,p}}
\\
&\qquad\lesssim
(n+1) \varsigma_n^{1/4}
\| (v_n \otimes v_n) \ast \chi_\ell - \mathring{R}_\ell + q_\ell Id \|_{C_{\mathfrak{t}}L_x^p}
\\
&\qquad\lesssim
(n+1) \varsigma_n^{1/4} \ell^{-\delta} \left(
\| v_n \|_{C_{\mathfrak{t}}L_x^2} +
\|\mathring{R}_n\|_{C_{\mathfrak{t}}L^1_x} + 
\| q_n \|_{C_{\mathfrak{t}}L^1_x} \right)
\\
&\qquad\lesssim
(n+1) \varsigma_n^{1/4} \ell^{-\delta} \lambda_n^{1/1000}
\lesssim
\delta_{n+3}^2.
\end{align*}
We point out that the previous inequalities hold true at least choosing $p$ sufficiently close to $1$, but again not depending on $a$, so that implicit constants can be absorbed taking $a$ large enough.

\subsubsection{Correction error}
It is sufficient to control, recalling previous bounds
\begin{align*}
\| w^{(p)}_{n+1} \otimes (w^{(c)}_{n+1}+w^{(t)}_{n+1}) \|_{C_{\mathfrak{t}}L_x^p}
&\lesssim
\| w^{(p)}_{n+1} \|_{C_{\mathfrak{t}}L_x^{2p}}
\| w^{(c)}_{n+1}+w^{(t)}_{n+1} \|_{C_{\mathfrak{t}}L_x^{2p}}
\\
&\lesssim
\ell^{-p'} (\delta_{n+3}^3 + \ell^{-22}\lambda_{n+1}^{-1/10} ) 
\lesssim
\delta_{n+3}^2,
\\ 
\|  (w^{(c)}_{n+1}+w^{(t)}_{n+1}) \otimes w_{n+1} \|_{C_{\mathfrak{t}}L_x^p}
&\lesssim
\| w_{n+1} \|_{C_{\mathfrak{t}}L_x^{2p}}
\| w^{(c)}_{n+1}+w^{(t)}_{n+1} \|_{C_{\mathfrak{t}}L_x^{2p}}
\\
&\lesssim
\ell^{-p'} (\delta_{n+3}^3 + \ell^{-22}\lambda_{n+1}^{-1/10} ) 
\lesssim
\delta_{n+3}^2.
\end{align*}
Putting all together, the proof of the bound \eqref{est:integrab} on $\|\mathring{R}_{n+1}\|_{C_\mathfrak{t}L_x^1}$ is proved up to noticing
\begin{align*}
\|\mathring{R}_{n+1}\|_{C_\mathfrak{t}L_x^1}
\lesssim
\|\mathring{R}_{n+1}\|_{C_\mathfrak{t}L_x^p}
\lesssim
\delta_{n+3}^2 + \lambda_{n+1}^{-1/10}
\leq
C_R \delta_{n+3}
\end{align*}
up to suitable choice of parameters.

\subsubsection{Estimate on $\| \mathring{R}_{n+1} \|_{C_\mathfrak{t}C_x^1}$}

This comes easily from the Navier-Stokes-Reynolds equation itself and the bounds on the derivatives of $v_{n+1}$, $q_{n+1}$ already proved in \autoref{ssec:velocity_est} and \autoref{ssec:oscillation}. Indeed, by \eqref{eq:R&q}
\begin{align*}
\mathring{R}_{n+1} 
\coloneqq
\mathcal{R}^{\phi_{n+1}}
\left( \partial_t v_{n+1} + \dvgphin(v_{n+1} \otimes v_{n+1}) + \nablaphin q_{n+1} - \Dphin v_{n+1}\right),
\end{align*}
and thus for every $\delta>0$ it holds
\begin{align*}
\| \mathring{R}_{n+1} \|_{C_\mathfrak{t}C^{1+\delta}_x}
&\lesssim
\| \partial_t v_{n+1} + \dvgphin(v_{n+1} \otimes v_{n+1}) + \nablaphin q_{n+1} - \Dphin v_{n+1} \|_{C_\mathfrak{t}C^{\delta}_x}
\\
&\lesssim
\| v_{n+1} \|_{C^1_{\mathfrak{t},x}}
\| v_{n+1} \|_{C^1_\mathfrak{t}C^3_x}
+
\| q_{n+1} \|_{C_\mathfrak{t}C^2_x}
\lesssim
\lambda_{n+1}^{18}.
\end{align*}

\subsection{Estimate on the pressure} \label{ssec:pressure_est}
As already explained in \autoref{ssec:oscillation}, in order to prove the $L_x^1$ estimate on the pressure we only need to bound $P$.

Let us therefore consider the auxiliary Navier-Stokes-Reynolds system 
\begin{align*}
\partial_t \tilde{u}_n 
+
\dvg (\tilde{u}_n \otimes \tilde{u}_n)
+
\nabla \tilde{p}_n
-
\Delta \tilde{u}_n
=
\dvg \mathring{\tilde{R}}_n,
\end{align*}
with $\dvg \tilde{u}_n = 0$ and $\int_{\T^3} \tilde{u}_n = 0$.
Notice that the differential operators above are not composed with the flow $\phi_{n}$.
Then, starting from the triple $(\tilde{u}_0,\tilde{p}_0,\mathring{\tilde{R}}_0) = 0$ and iterating the system $(\tilde{u}_n,\tilde{p}_n,\mathring{\tilde{R}}_n)$ \emph{simultaneously} with the iteration for $({v}_n,{p}_n,\mathring{{R}}_n)$ according to
\begin{align*}
\tilde{u}_{n+1}
:=
\tilde{u}_\ell + \tilde{w}_{n+1}
:=
\tilde{u}_\ell + \tilde{w}^{(p)}_{n+1} + \tilde{w}^{(c)}_{n+1}+ \tilde{w}^{(t)}_{n+1},
\end{align*}
where for every $n \in \N$ we define
\begin{align*}
\tilde{w}^{(p)}_{n+1} 
&:=
\sum_{\xi \in \Lambda} a_{\xi} W_{\xi},
\\
\tilde{w}^{(c)}_{n+1}
&:=  \sum_{\xi \in \Lambda}
\curl(\nabla a_{\xi} \times V_{\xi}) + \nabla a_{\xi} \times \curl V_{\xi} + a_{\xi} W^{(c)}_{\xi},
\\
\tilde{w}^{(t)}_{n+1}
&:=
- \frac{1}{\mu} \mathcal{P}
\Pi_{\neq 0}
\sum_{\xi \in \Lambda} 
a_{\xi}^2 \theta_{\xi}^2 \psi_{\xi}^2 \xi,
\end{align*}
and
\begin{align*}
\tilde{p}_{n+1} 
&:=
\tilde{p}_\ell
-
P,
\\
\mathring{\tilde{R}}_{n+1}
&:=
\mathcal{R}
\left(
\partial_t \tilde{v}_{n+1} 
+
\dvg (\tilde{u}_{n+1} \otimes \tilde{u}_{n+1})
+
\nabla \tilde{p}_{n+1} 
-
\Delta \tilde{u}_{n+1}
\right),
\end{align*}
one can prove the bounds 
\begin{align*}
\| \tilde{u}_{n+1}-\tilde{u}_n \|_{C_\mathfrak{t}L_x^2}
\leq
C_u \delta_{n+1}^{1/2},
\quad
\| \tilde{u}_n \|_{C^1_{\mathfrak{t},x}}
\leq
C_u \lambda_n^{6},
\\
\| \mathring{\tilde{R}}_n \|_{C_\mathfrak{t}L_x^1} \leq
C_R 3^n,
\quad
\| \mathring{\tilde{R}}_n \|_{C_\mathfrak{t}C^1_x} \leq
C_R \lambda_n^{20},
\end{align*}
for every $n \in \N$.
Indeed, we have used $\rho$ and $a_\xi$ exactly as before, and so they are constructed from $\mathring{{R}}_\ell$ and not $\mathring{\tilde{R}}_\ell$; in particular, the bounds on $\tilde{u}_n$ and $\mathring{\tilde{R}}_n$ in $C^1$ as well as the bound on the increment $\tilde{u}_{n+1}-\tilde{u}_n$ are readily proved.

The only slight difference is the estimate on $\mathring{\tilde{R}}_n$ in $L^1_x$, for which we argue as follows.
First, one has the following decomposition of $\dvg \mathring{\tilde{R}}_{n+1}$
\begin{align*}
\partial_t \tilde{u}_{n+1} 
&+ 
\dvg (\tilde{u}_{n+1}\otimes \tilde{u}_{n+1})
+
\nabla \tilde{p}_{n+1}
-
\Delta \tilde{u}_{n+1}
\\
&= \nonumber
\underbrace{\left[\partial_t \tilde{w}^{(p)}_{n+1} + \partial_t \tilde{w}^{(c)}_{n+1} + \dvg (\tilde{w}_{n+1} \otimes \tilde{u}_\ell + \tilde{u}_\ell \otimes \tilde{w}_{n+1} ) - \Delta \tilde{w}_{n+1} \right]}_{=\,linear\, error}
\\
&\quad+ \nonumber
\underbrace{\left[\dvg \left(\tilde{w}^{(p)}_{n+1} \otimes \tilde{w}^{(p)}_{n+1} + \mathring{R}_\ell \circ \phi_{n+1}^{-1} - \rho \circ \phi_{n+1}^{-1} Id \right) + \partial_t \tilde{w}^{(t)}_{n+1} + \nabla P \right]}_{=\,oscillation\, error}
\\
&\quad+ \nonumber
\underbrace{\left[ \dvg (\tilde{u}_\ell \otimes \tilde{u}_\ell - (\tilde{u}_n \otimes \tilde{u}_n) \ast \chi_\ell)\right]}_{=\,mollification\, error}
\\
&\quad+ \nonumber
\underbrace{\left[\dvg ( \tilde{w}^{(p)}_{n+1} \otimes (\tilde{w}^{(c)}_{n+1}+\tilde{w}^{(t)}_{n+1}) + (\tilde{w}^{(c)}_{n+1}+\tilde{w}^{(t)}_{n+1}) \otimes \tilde{w}_{n+1}) \right]}_{=\,corrector\, error}
\\
&\quad+ \nonumber
\underbrace{\left[\dvg \left( \rho \circ \phi_{n+1}^{-1} Id -
\mathring{R}_\ell \circ \phi_{n+1}^{-1} + \mathring{\tilde{R}}_\ell  \right) \right]}_{=\,remainder\, error}.
\end{align*} 
Notice that
\begin{align*} 
w^{(p)}_{n+1} \otimes w^{(p)}_{n+1} 
&=
\sum_{\xi \in \Lambda} a_{\xi}^2\, \Pi_{\neq 0} (W_{\xi} \otimes W_{\xi})
+ 
\rho \circ \phi_{n+1}^{-1} Id
-
\mathring{R}_\ell \circ \phi_{n+1}^{-1},
\end{align*}
and hence the fast-fast interaction of the perturbation is not compensated by $\mathring{\tilde{R}}_\ell$ but rather by the ``old" terms $\rho \circ \phi_{n+1}^{-1} Id + \mathring{{R}}_\ell \circ \phi_{n+1}^{-1}$, for which however we have bounds in $L_x^p$.
All the other terms are controlled as usual, and therefore a remainder error proportional to 
\begin{align*}
\|\rho \circ \phi_{n+1}^{-1} Id -
\mathring{R}_\ell \circ \phi_{n+1}^{-1} + \mathring{\tilde{R}}_\ell\|_{L_x^p}
\leq 2 C_R 3^n
\end{align*} 
in the $C_\mathfrak{t}L^1_x$ norm of $\mathring{\tilde{R}}_{n+1}$ accumulates along the iteration.
Notice that the auxiliary Reynolds stress $\mathring{\tilde{R}}_{n+1}$ is not small, but this is not important for our purposes since we only use it to recover estimates on $P$.
Indeed, we know 
\begin{align*}
\partial_t \tilde{u}_n 
+
\dvg (\tilde{u}_n \otimes \tilde{u}_n)
+
\nabla \tilde{p}_n
-
\Delta \tilde{u}_n
&=
\dvg \mathring{\tilde{R}}_n,
\\
\partial_t \tilde{u}_{n+1} 
+
\dvg (\tilde{u}_{n+1} \otimes \tilde{u}_{n+1})
+
\nabla \tilde{p}_{n+1} 
-
\Delta \tilde{u}_{n+1}
&= 
\dvg \mathring{\tilde{R}}_{n+1},
\end{align*}
and $ \dvg \tilde{u}_n = \dvg \tilde{u}_{n+1}  = 0$ (the perturbation $\tilde{w}_{n+1}$ is divergence free).
Thus
\begin{align} \label{eq:Delta_pressure_q}
\Delta \tilde{p}_n 
&=
\dvg\dvg (\mathring{\tilde{R}}_n- \tilde{u}_n \otimes \tilde{u}_n),
\\ \label{eq:Delta_pressure_q+1}
\Delta \tilde{p}_{n+1} 
&=
\dvg\dvg (\mathring{\tilde{R}}_{n+1}- \tilde{u}_{n+1} \otimes \tilde{u}_{n+1}),
\end{align}
and Schauder estimates imply for every $p>1$ sufficiently small (independent of $n$)
\begin{align*}
\| \tilde{p}_n \|_{L_x^1}
\lesssim
\| \tilde{p}_n \|_{L_x^p}
&\lesssim
\| \mathring{\tilde{R}}_n \|_{L_x^p}
+
\| \tilde{u}_n \|_{L_x^{2p}}^2
\\
&\lesssim
\| \mathring{\tilde{R}}_n \|_{L_x^p}
+
\sum_{k=1}^{n} \| \tilde{u}_{k} - \tilde{u}_{k-1} \|_{L_x^{2p}}^2
\leq \tfrac12 \lambda_n^{1/1000},
\end{align*}
and similarly for $n+1$. 
In addition, we know that $P$ is related to $\tilde{p}_n,\tilde{p}_{n+1}$ via the formula
\begin{align*}
\tilde{p}_{n+1} = \tilde{p}_\ell  - P, 
\end{align*}
so that we have
\begin{align*}
\| P \|_{L_x^1} 
\leq 
\| \tilde{p}_{n+1} \|_{L_x^1}
+
\| \tilde{p}_n \|_{L_x^1}
\leq 
\lambda_{n+1}^{1/1000},
\end{align*}
concluding the proof.

\begin{rmk}
We critically need \eqref{eq:Delta_pressure_q} and \eqref{eq:Delta_pressure_q+1} for the ``deterministic" pressures $\tilde{p}_n$, $\tilde{p}_{n+1}$.
Indeed, it does not seem possible to recover directly good estimates on $P$, and thus any estimate on the pressure $q_{n+1}$ must be obtained from the equation itself.
However, in the stochastic case this is not feasible: indeed
\begin{align*}
\Dphin q_{n+1}
=
\dvgphin \dvgphin (\mathring{R}_{n+1} - v_{n+1} \otimes v_{n+1} + \nablaphin v_{n+1} )
-
\dvgphin (\partial_t v_{n+1}),
\end{align*} 
where the additional terms come from the fact that $\dvgphin v_{n+1} \neq 0$ in general, and moreover $\dvgphin$ and $\partial_t$ do not commute.
But the last term in the previous inequality prevents us from obtaining any good bound on $\| q_{n+1} \|_{L_x^1}$, since we can only estimate $\|\partial_t v_{n+1} \|_{W_x^{-1,1}}$ with $\|v_{n+1}\|_{C^1_{\mathfrak{t},x}} \lesssim \lambda_{n+1}^{6}$.
\end{rmk} 

\section{Construction of solutions with prescribed initial condition}
\label{sec:convex_u0}

In this second part we want to modify the previous construction so to prescribe the initial value of the solution to \eqref{eq:NS}.
Let $u_0 \in L_x^2$ with $\dvg u_0 = 0$ in the sense of distributions.
We shall assume $u_0$ independent of the Brownian motion $B$ and replace the filtration $\{\mathcal{F}_t\}_{t \geq 0}$ with the augmented canonical filtration generated by $(u_0,B)$.
Moreover, without loss of generality we can suppose that 
\begin{align*}
\| u_0 \|_{L_x^2} 
\leq 
M,
\quad
\mathbb{P}-\mbox{almost surely} 
\end{align*}
for some finite constant $M$. 
Indeed, for a general initial condition $u_0 \in L_x^2$ almost surely one defines $\Omega_M := \{ M-1 \leq \| u_0 \|_{L_x^2} < M\}$.
Then, given the existence of infinitely many solutions $u_M$ on each $\Omega_M$ one can define $u := \sum_{M \in \N, M \geq 1} u_M \mathbf{1}_{\Omega_M}$, solving the equation with initial condition $u_0$.

In order to prescribe the initial value of solutions we follow the approach of \cite{BuMoSz20} and \cite{HoZhZh22b+}.
However, differently from what done in \autoref{sec:convex_energy} here we need to use the so-called \emph{Da Prato-Debussche trick} to produce a good initial triple $(v_0,q_0,\mathring{R}_0)$ to start the iteration with.
 
Formally, let $Z$ be the solution of the Stokes system with transport noise emanating from $u_0$:
\begin{align} \label{eq:Z}
\begin{cases}
d Z  + \sum_{k \in I} (\sigma_k \cdot \nabla) Z \bullet dB^k + \nabla p_Z dt = \Delta Z dt,
\\
\dvg Z = 0,
\\
Z|_{t=0} = u_0,
\end{cases}
\end{align}
and let $u$ be any solution of \eqref{eq:NS} with the same initial condition. Then the difference $V := u-Z$ solves
\begin{align} \label{eq:V}
\begin{cases}
d V  + \dvg \left( (V+Z) \otimes (V+Z) \right) dt
+
\sum_{k \in I} (\sigma_k \cdot \nabla) V \bullet dB^k + \nabla p_V \,dt = \Delta V dt,
\\
\dvg V = 0,
\\
V|_{t=0} = 0,
\end{cases}
\end{align}
which is equivalent, by composition with the flow $(v,q,z) := (V,p_V,Z) \circ \phi$, to the system
\begin{align} \label{eq:DaP-D}
\begin{cases}
\partial_t v  + \dvgphi \left( (v+z) \otimes (v+z) \right) 
+ \nablaphi q = \Dphi v,
\\
\dvgphi v  = 0,
\\
v|_{t=0} = 0.
\end{cases}
\end{align}
Viceversa, via the inverse transformation $u=(v+z) \circ \phi^{-1}$, any solution of \eqref{eq:DaP-D} yields a solution to \eqref{eq:NS} satisfying $u|_{t=0}=u_0$.
Therefore, in order to construct (multiple) solutions of \eqref{eq:DaP-D} one can introduce for every $n\in \N$ a smooth approximation $z_n$ of $z$ and solve the Navier-Stokes-Reynolds system on $t \in \R$
\begin{align} \label{eq:random_NSbis}
\begin{cases}
\partial_t v_n  + \dvgphien \left( (v_n+z_n) \otimes (v_n+z_n) \right) 
+ \nablaphien q_n  - \Dphien v_n
=
\dvgphien \mathring{R}_n,
\\
v_n|_{t = 0} = 0.
\end{cases}
\end{align}
Then, a solution to \eqref{eq:DaP-D} is obtained showing the convergences $v_n \to v$, $z_n \to z$ and $\mathring{R}_n$, $\dvgphien v_n \to 0$ in suitable spaces as $n \to \infty$. 

\begin{rmk}
Before moving on, let us first comment on the Da Prato-Debussche trick.
Usually, it is used in semilinear stochastic equations with additive noise to ``remove" the stochastic integral; this is obtained taking the difference between any hypothetical solution of the original equation and the stochastic convolution, and observing that the difference solves a random PDE.
In our case, however, since the noise is multiplicative we still have a stochastic integral in \eqref{eq:V}, and we need the flow transformation anyway to reformulate the problem as a random PDE.  
The point is that here we use the Da Prato-Debussche trick only to produce a good initial triple $(v_0,q_0,\mathring{R}_0)$. 
More specifically, the iteration will be started at $v_0 = q_0 = 0$ and $\mathring{R}_0 = z_0 \otimes z_0$, extended to negative times with their value at time $t=0$; it is in fact this extension to negative times that we are not able to perform otherwise. 
\end{rmk}

\subsubsection{Approximation of $z$}
Let us first study the Stokes system \eqref{eq:Z}. 
Via Galerkin approximation and compactness, it is easy to  obtain a probabilistically weak, analytically weak, progressively measurable solution $Z$ in $L^\infty (\Omega, L^\infty_t L_x^2 \cap L^2_t H_x^1)$ satisfying almost surely
\begin{align*}
\| Z(t) \|_{L_x^2}^2 + 2 \int_0^t \| Z(s) \|_{H_x^1}^2 ds
\leq
\| u_0 \|_{L_x^2}^2
\leq
M^2. 
\end{align*}
Moreover, pathwise uniqueness for the equation in the space $L^2 (\Omega, L^\infty_t L_x^2 \cap L^2_t H_x^1)$ descends for instance from application of \cite[Proposition A.1]{DeHoVo16} with $\psi = 1$ and $\varphi(u) = u^2$.
Thus $Z$ is a probabilistically strong solution by Yamada-Watanabe Theorem.

Moreover, by It\=o isometry one can prove for every $p<\infty$ 
\begin{align*}
\mathbb{E} \left\|Z(t)-Z(s)\right\|_{H_x^{-1}}^p
\lesssim
M^p |t-s|^{p/2},
\end{align*}
where the implicit constant depends only on $p$, and thus by Kolmogorov continuity criterion we can assume without loss of generality that $Z$ is almost surely $\gamma$-H\"older continuous for every $\gamma<1/2$, as a process taking values in $H^{-1}_x$. Thus, modifying the definition of our stopping times $\mathfrak{t}_L$ from \autoref{lem:flow} we can assume 
\begin{align*}
\|Z\|_{C^{1/4}_{\mathfrak{t}_L} H_x^{-1}} \leq C_Z L M
\end{align*} 
almost surely for some $C_Z \in (0,\infty)$. 
Similarly to what previously done in \autoref{sec:convex_energy}, for simplicity we only work up to time $\mathfrak{t}=\mathfrak{t}_1$ hereafter.
Furthermore, in this section we shall suppose $\mathfrak{t} \geq 1$ with probability larger than $1/2$ (this is always possible modifying the value of $C_Z$) and also $\mathfrak{t} \leq 2$ almost surely.

We approximate $z$ via Fourier projection in space, $z_n := \Pi_{\leq \kappa_n} z$ for some frequency cut-off $\kappa_n := \lambda_{n+1}^{\alpha/16} \to \infty$.
By Sobolev embedding we have the estimates (also use arguments as in \cite[Lemma C.2]{HoLaPa22+} to deduce $\|z\|_{C^{1/4}_{\mathfrak{t}} H_x^{-1}} \lesssim \|Z\|_{C^{1/4}_{\mathfrak{t}} H_x^{-1}}$)
\begin{align*}
\| z_n \|_{L^\infty_{\mathfrak{t},x}} 
\lesssim 
M\kappa_n^2,
\qquad
\| z_n \|_{L^\infty_\mathfrak{t}C_x^2} 
&\lesssim 
M\kappa_n^4,
\qquad
\|z_n\|_{C^{1/4}_{\mathfrak{t},x}} 
\lesssim
M\kappa_n^3.
\end{align*}
We shall also need a control on $z_{n+1}-z_n$ in $L^2_x$. Using $\|Z\|_{L^2_\mathfrak{t}H^1_x} \lesssim M$ we have
\begin{align*}
\| z_{n+1} - z_n \|_{L^2_tL^2_x} 
&\lesssim 
M
\kappa_n^{-1}.
\end{align*}
Notice that the previous bound is not uniform with respect to time. As a consequence, in this section we will iteratively control the Reynolds stress $\mathring{R}_n$ in $L^2_\mathfrak{t} L^1_x$ and the velocity $v_n$ in $L^4_\mathfrak{t} L^2_x$, in contrast with uniform-in-time controls of \autoref{sec:convex_energy}.

\subsection{Iterative assumptions}

Let the parameters $\lambda_n$, $\delta_n$, $r_\perp$, $\rpar$, $\mu$, $\ell$ and $\varsigma_n$ be given by the same formulas as in \autoref{sec:convex_energy}, but with possibly different values of the parameters $a$, $b$, $\alpha$, $\beta$.

Let $\gamma_n := 2^{-n}$ for $n \in \N \setminus \{n_0\}$ and $\gamma_{n_0} \in \{0,1\}$ be given, for some $n_0 \in \N$. 
Also, introduce a sequence of times $\tau_n := 2^{-n}$ and assume in addition $\ell \leq \varsigma_n \leq \tau_{n+1}$.
In this section, the following iterative assumptions shall be in force for some $C_v, C_q, C_R \in (0,\infty)$.

First, an assumption on the $L_x^2$ and $W_x^{1,1}$ norms of $v_n$
\begin{align} \label{est2:v} 
v_n
&=
0,
\quad
\forall t \in [-\tau_n,\tau_n \wedge \mathfrak{t}],
\\
{
\int_{-\tau_n}^\mathfrak{t}
\| v_n(t) \|^4_{L_x^2} dt}
&\leq
6 C_v^4 (n+1)^2,
\\ 
\| v_n(t) \|_{W_x^{1,1}}
&\leq
C_v \sum_{k=0}^n \delta_{k+1}^2,
\quad
\forall t \in (\tau_n \wedge \mathfrak{t} , \mathfrak{t} ].
\end{align}

Also, a control on the derivatives
\begin{gather} \label{est2:deriv}
\| v_n \|_{C^1_{\mathfrak{t},x}} 
\leq
C_v  \lambda_n^{6},
\qquad
\| v_n \|_{C^1_{\mathfrak{t}}C_x^3} 
\leq
C_v  \lambda_n^{12},
\\ \label{est2:deriv_qR}
\| q_n \|_{L^\infty_\mathfrak{t} C^2_x} 
\leq
C_q  \lambda_n^{12},
\qquad
\| \mathring{R}_n \|_{L^\infty_\mathfrak{t} C^1_x} 
\leq
C_R \lambda_n^{20} \lambda_{n+1}^{\alpha/4}.
\end{gather}

The Reynolds stress and the pressure need to be controlled in $L^2_t L^1_x$. We have
\begin{align}
\int_{-\tau_n}^{\tau_{n-1} \wedge \mathfrak{t}}
\| \mathring{R}_n(t) \|_{L_x^1}^2 dt
&\leq \label{est2:RinL1small}
C_R^2 (n+1)^2
\\
\int_{\tau_{n-1} \wedge \mathfrak{t}}^{\mathfrak{t}}
\| \mathring{R}_n(t) \|_{L_x^1}^2 dt
&\leq \label{est2:RinL1large}
C_R^2 \delta_{n+1}^2,
\\
\int_{-\tau_n}^{\mathfrak{t}}
\| q_n(t) \|_{L_x^1}^2 dt
&\leq \label{est2:press}
C_q^2 \lambda_n^{1/500}.
\end{align}

Finally, the estimates on the divergence
\begin{align} \label{est2:diverg}
{\| \dvgphien v_n \|_{L^4_\mathfrak{t} H_x^{-1}}}
\leq 
C_v  
\delta_{n+2}^3,
\qquad
\| \dvgphien v_n \|_{L^\infty_\mathfrak{t} L_x^1}
\leq 
C_v  
\delta_{n+2}^3.
\end{align}

\begin{prop} \label{prop:it_bis}
There exist a choice of parameters $a,b,\alpha,\beta$ and a constant $C_e \in (0,\infty)$ with the following property.
Let $(v_n,z_n,q_n,\phi_n,\mathring{R}_n)$, $n \in \N$ be a solution of \eqref{eq:random_NSbis} with $z_n = \Pi_{\leq \kappa_n} z$ and $\phi_n$ given by \eqref{eq:def_phi_n}, and satisfying the inductive estimates \eqref{est2:v} to \eqref{est2:diverg}.
Then there exists a quintuple $(v_{n+1},z_{n+1},q_{n+1},\phi_{n+1},\mathring{R}_{n+1})$, solution of \eqref{eq:random_NSbis} with $z_{n+1} = \Pi_{\leq \kappa_{n+1}} z$ and $\phi_{n+1}$ given by \eqref{eq:def_phi_n}, satisfying the same inductive estimates with $n$ replaced by $n+1$ and such that
\begin{align} 
\int_{\tau_{n-2} \wedge \mathfrak{t}}^\mathfrak{t}
\| v_{n+1}(t)-v_n(t) \|_{L^2_x}^4 dt
&\leq \label{ass:it_vL_uno}
C_v^4 \gamma_{n+1}^2,
\\
\int_{\tau_{n+1} \wedge \mathfrak{t}}^{\tau_{n-2} \wedge \mathfrak{t}}
\| v_{n+1}(t)-v_n(t) \|_{L^2_x}^4 dt
&\leq \label{ass:it_vL_due}
C_v^4 (n+1)^2,
\\
\|v_{n+1}(t)-v_n(t) \|_{L_x^2}
&= \label{ass:it_vL_tre}
0,
\quad
\forall t \in [-\tau_{n+1}, \tau_{n+1} \wedge \mathfrak{t}],
\\
\|v_{n+1}(t)-v_n(t) \|_{W_x^{1,1}} \label{ass:it_vW}
&\leq
C_v \delta_{n+2}^2,
\quad
\forall t \in [-\tau_{n+1}, \mathfrak{t}],
\\
\label{ass:energy_pump}
\int_{\tau_{n-2} \wedge \mathfrak{t}}^\mathfrak{t}
| \| v_{n+1} (t)\|_{L_x^2}^2 -\| v_n (t)\|_{L_x^2}^2 - 3 \gamma_{n+1} |\, dt 
&\leq 
C_e \delta_{n+1}.
\end{align}
\end{prop}

\begin{proof}[Proof of \autoref{thm:u_0}]
We apply \autoref{prop:it_bis} iteratively, starting the iteration from $v_0 = q_0 = 0$, $\mathring{R}_0 = z_0 \otimes z_0$.
Recall that by definition $z_0 = \Pi_{\leq 1} z$, so that $\| \mathring{R}_0 \|_{L^\infty_\mathfrak{t}C^1_x} \lesssim \|u_0\|_{L_x^2}^2 \leq M^2$ and the iterative assumptions hold true.
Assumptions \eqref{ass:it_vL_uno}-\eqref{ass:it_vL_due}-\eqref{ass:it_vL_tre} (resp. \eqref{ass:it_vW}) guarantee that $\{ v_n\}_{n \in \N}$ is a Cauchy sequence in $L^p_\mathfrak{t} L^2_x$ for every {$p < 4$} (resp. in $C_\mathfrak{t}W^{1,1}_x$). Denote $v := \lim_{n \to \infty} v_n$, which satisfies $v|_{t=0} = 0$ by \eqref{est2:v} and is progressively measurable. 
By \eqref{est2:RinL1small} and \eqref{est2:RinL1large} we deduce $\mathring{R}_n \to 0$ in $L^1_{\mathfrak{t},x}$. Taking into account also \eqref{est2:diverg} and $z_n \to z$ in $L^2_{\mathfrak{t},x}$ we deduce that $v$ is a solution of \eqref{eq:DaP-D} on the time interval $[0,\mathfrak{t}]$. 
The regularity claimed in the statement of the theorem holds true for $u = (v+z) \circ \phi^{-1}$ since with probability one $v \in L^p_\mathfrak{t} L^2_x \cap C_\mathfrak{t}W^{1,1}_x$ for every $p<4$ and $z \in L^\infty_\mathfrak{t} L^2_x \cap L^2_\mathfrak{t} H^1_x \cap C^{1/4}_\mathfrak{t} H^{-1}_x$.
General time intervals $[0,\mathfrak{t}_L]$ are dealt with including a factor $L_n=L^{m^{n+1}}$ in the iterative assumptions, $m<b$. 

Next, we prove non-uniqueness of solutions.
Choose $n_0 \geq 4$ such that
\begin{align*}
\int_{\tau_{n_0-3} \wedge \mathfrak{t}}^\mathfrak{t}
| \|v_{n_0}(t)\|^2_{L_x^2} - \|v_{n_0-1}(t)\|^2_{L_x^2} - 3\gamma_{n_0} | dt
\leq 1/3.
\end{align*}
This distinguishes the two solutions $v^1,v^2$ obtained imposing $\gamma_{n_0}=1$ and $\gamma_{n_0}=0$ respectively. Indeed, $\mathfrak{t} > 1$ with probability at least $1/2$, and thus the length of the time interval $[\tau_{n_0-3} \wedge \mathfrak{t}, \mathfrak{t}]$ is at least $1/2$ (since $n_0 \geq 4$) with probability at least $1/2$. Thus it must be
\begin{align*}
\left| \int_{\tau_{n_0-3} \wedge \mathfrak{t}}^{\mathfrak{t}} \left(  
 \|v^1(t)\|^2_{L_x^2} - \|v^1_{n_0-1}(t)\|^2_{L_x^2} - 3 \right) dt \right|
&\leq
1/3 + |\mathfrak{t}| \sum_{n > n_0} \gamma_n
\leq
2/3,
\\
\left| \int_{\tau_{n_0-3} \wedge \mathfrak{t}}^{\mathfrak{t}} \left(  
 \|v^2(t)\|^2_{L_x^2} - \|v^2_{n_0-1}(t)\|^2_{L_x^2} \right) dt \right|
&\leq
1/3 + |\mathfrak{t}| \sum_{n > n_0} \gamma_n
\leq
2/3,
\end{align*}
where we used $\mathfrak{t} \leq 2$ almost surely. Since $v^1_{n_0-1}=v^2_{n_0-1}$ we have
\begin{align*}
\int_{\tau_{n_0-3} \wedge \mathfrak{t}}^{\mathfrak{t}} \|v^1(t)\|^2_{L_x^2} dt 
\geq  
-2/3 + 3/2 + \int_{\tau_{n_0-3} \wedge \mathfrak{t}}^{\mathfrak{t}} \|v^1_{n_0-1}(t)\|^2_{L_x^2} dt 
\geq 
-4/3 + 3/2 +
\int_{\tau_{n_0-3} \wedge \mathfrak{t}}^{\mathfrak{t}} \|v^2(t)\|^2_{L_x^2} dt 
\end{align*}
with probability at least $1/2$.
\end{proof}

\subsection{Convex integration scheme}
In the remainder of the section we prove \autoref{prop:it_bis}. 
Since the construction is similar to that of \autoref{sec:convex_energy}, many estimates will be similar to those already seen before, and they will be omitted when possible.

\subsubsection{Perturbation of the velocity}

Following \cite{HoZhZh22b+}, define
\begin{align*}
\rho 
&:= 
2 \sqrt{\ell^2 + |\mathring{R}_\ell|^2} + \frac{\gamma_{n+1}}{(2\pi)^3}.
\end{align*} 
The amplitude functions $\mathbf{a}_\xi$ and $a_\xi$ are defined accordingly via the formula \eqref{eq:amplitude}, as well as the perturbations $w^{(p)}_{n+1}$, $w^{(c)}_{n+1} := w^{(c,1)}_{n+1} + w^{(c,2)}_{n+1}$ and $w^{(t)}_{n+1}$, formulas \eqref{eq:def_wp}, \eqref{eq:def_wc1}, \eqref{eq:def_wc2}, and \eqref{eq:def_wt} respectively.
We point out that, although we use the same symbols as in \autoref{sec:convex_energy}, here the objects are different because $\rho$ and $\mathring{R}_n$ are different. 
Then, given a smooth non-decreasing cut-off $\chi:\R \to [0,1]$ identically equal to $0$ on $(-\infty,\tau_{n+1}]$ and to $1$ on $[\tau_n, \infty)$, define
\begin{align*}
\tilde{w}^{(p)}_{n+1}
:=
{w}^{(p)}_{n+1} \chi,
\quad
\tilde{w}^{(c)}_{n+1}
:=
{w}^{(c)}_{n+1} \chi,
\quad
\tilde{w}^{(t)}_{n+1}
:=
{w}^{(t)}_{n+1} \chi^2.
\end{align*}

\subsubsection{Estimates on the velocity}

By definition we have for every $t \leq \mathfrak{t}$
\begin{align} \label{eq:rhoL1}
\| \rho (t) \|_{L_x^1} 
&\lesssim
\ell 
+
\| \mathring{R}_\ell(t) \|_{L_x^1}
+
\gamma_{n+1},
\\
\|a_\xi(t)\|_{L_x^2}  \label{eq:aL2}
&=
\|\mathbf{a}_\xi(t)\|_{L_x^2}
\lesssim
\| \rho (t)\|_{L_x^1}^{1/2},
\end{align} 
with universal implicit constant.
By Sobolev embedding $W^{5,1}_{\mathfrak{t},x} \subset C_{\mathfrak{t},x}$ and the bounds on $\mathring{R}_n$ in $L^2_\mathfrak{t} L^1_x$ given by \eqref{est2:RinL1small} and \eqref{est2:RinL1large} then we have 
\begin{align}
%
%
%
\| a_\xi \|_{C_\mathfrak{t} C_x}^2 \label{eq:aCsmall}
&=
\| \mathbf{a}_\xi \|_{C_\mathfrak{t}C_x}^2
\lesssim 
\| \rho \|_{C_\mathfrak{t}C_x} 
\lesssim 
\ell^{-5} (n+1),
\end{align}
whereas by (3.25) and (3.34) therein we have for every $N = 1,2,...10$ and $M=0,1$
\begin{align}
\| \rho \|_{C^N_{\mathfrak{t},x}} 
&\lesssim  \nonumber
\ell^{2-8N} (n+1),
\\
\| a_\xi \|_{C^M_{\mathfrak{t}}C^N_x} 
&\lesssim \label{eq:aCN}
\varsigma_{n+1}^{-M}
\| \mathbf{a}_\xi \|_{C^{M+N}_{\mathfrak{t},x}}
\lesssim
\varsigma_{n+1}^{-M} \ell^{-10-8(M+N)} (n+1)^{1/2}.
\end{align}

The different exponents of $\ell$ in the previous lines with respect to those in \autoref{sec:convex_energy} are due to the fact that our iterative assumption on $\mathring{R}_n$ is a bound in $L^2_\mathfrak{t}L^1_x$ instead of $C_\mathfrak{t}L^1_x$.

Let us now give estimates for the velocity increments.
For $t \in [-\tau_{n+1},\tau_{n+1} \wedge \mathfrak{t}]$ we have $\tilde{w}_{n+1} (t) \equiv 0$ because of the cut-off, and therefore \eqref{ass:it_vL_tre} holds.

The principal part of the perturbation $\tilde{w}^{(p)}_{n+1}$ is controlled in $L^2_x$ as follows. By \eqref{eq:rhoL1}, \eqref{eq:aL2}, \eqref{eq:aCN} and \cite[Lemma 7.4]{BuVi19b} one has for every $t \leq \mathfrak{t}$
\begin{align*}
\| \tilde{w}_{n+1}^{(p)} (t)\|_{L_x^2}
=
\| (\tilde{w}_{n+1}^{(p)}\circ \phi_{n+1}^{-1}) (t)\|_{L_x^2}
\lesssim
\| \mathring{R}_{\ell}(t) \|_{L^1_x}^{1/2} 
+
\gamma_{n+1}^{1/2}.
\end{align*}
In particular, by assumptions \eqref{est2:RinL1small} and \eqref{est2:RinL1large} it holds for some universal $C_v \in (0,\infty)$
\begin{align*}
\int_{\tau_{n-2} \wedge \mathfrak{t}}^\mathfrak{t}
\| \tilde{w}_{n+1}^{(p)}(t) \|_{L^2_x}^4 dt
&\leq \frac{C_v^4}{2}
\gamma_{n+1}^2,
\\
\int_{\tau_{n+1} \wedge \mathfrak{t}}^{\tau_{n-2} \wedge \mathfrak{t}}
\| \tilde{w}_{n+1}^{(p)}(t) \|_{L^2_x}^4 dt
&\leq \frac{C_v^4}{2}
(n+1)^2.
\end{align*} 

On the other hand, using \eqref{eq:aCsmall} we have for $p \in (1,\infty)$ and $t \leq \mathfrak{t}$ (recall that $\delta_{n+1} \lesssim \gamma_{n+1} \lesssim n+1$ for every $n$)
\begin{align*}
\| \tilde{w}^{(p)}_{n+1}(t) \|_{L_x^p}
&\lesssim
(n+1)^{1/2}
\ell^{-10} r_\perp^{2/p-1} \rpar^{1/p-1/2},
\\
\| \tilde{w}^{(c,2)}_{n+1}(t) \|_{L_x^p}
&\lesssim
(n+1)^{1/2}
\ell^{-26} r_\perp^{2/p} \rpar^{1/p-3/2},
\\
\| \tilde{w}^{(t)}_{n+1}(t) \|_{L_x^p}
&\lesssim
(n+1)\,
\ell^{-20} r_\perp^{2/p-1} \rpar^{1/p-2} \lambda_{n+1}^{-1}.
\end{align*}

Finally, $\|w^{(c,1)}_{n+1}\|_{L^4_\mathfrak{t}L_x^2} \lesssim \| \dvgphien v_n \|_{L^4_\mathfrak{t}H_x^{-1}} \lesssim \delta_{n+2}^3$. Thus \eqref{ass:it_vL_uno} and \eqref{ass:it_vL_due} hold true taking $p=2$ above and $a$ sufficiently large.

Also, arguing as in \autoref{sec:convex_energy} it is easy to check \eqref{ass:it_vW} and \eqref{est2:deriv} (all the additional factors $\ell^{-1}$, as well as the time derivative of the cut-off $|\chi'|\lesssim 2^n$ can be absorbed into some positive power of $\lambda_{n+1}$).
Moreover, the estimates on the divergence \eqref{est2:diverg} descend from the bounds on $\|v_n\|_{L^4_\mathfrak{t}L_x^2}$, $\|v_n\|_{L^\infty_\mathfrak{t}W_x^{1,1}}$ as in \autoref{sec:convex_energy}, and are omitted.

\subsubsection{Estimate on the energy}

The energy estimate \eqref{ass:energy_pump}, contrary to the corresponding assumption \eqref{est:energy} of \autoref{sec:convex_energy}, does not provide a prescribed energy profile for the velocity field; instead, it serves to quantify how much energy is pumped into the system with the perturbation $\tilde{w}_{n+1}$.
Incidentally, tuning the parameter $\gamma_{n_0}$ as in the proof of \autoref{thm:u_0}, this gives non-uniqueness of solutions.

Let $t \in [0,\mathfrak{t}]$ be given. By the same computations as in \autoref{sec:convex_energy}, we have
\begin{align*}
| \|v_{n+1}(t)\|^2_{L_x^2} - \|v_n(t)\|^2_{L_x^2} - 3\gamma_{n+1} | \lesssim \delta_{n+1} + \|\mathring{R}_\ell(t)\|_{L_x^1}. 
\end{align*}
Therefore, since we are assuming $\mathfrak{t} \leq 2$ almost surely and using \eqref{est2:RinL1large}, we obtain
\begin{align*}
\int_{\tau_{n-2} \wedge \mathfrak{t}}^\mathfrak{t}
| \|v_{n+1}(t)\|^2_{L_x^2} - \|v_n(t)\|^2_{L_x^2} - 3\gamma_{n+1} | dt
\lesssim
\delta_{n+1},
\end{align*}
where the implicit constant is universal (denoted $C_e$ in \eqref{ass:energy_pump}).

\subsubsection{The pressure $q_{n+1}$ and Reynolds stress $\mathring{R}_{n+1}$}
We shall define the new pressure as 
\begin{align*}
q_{n+1} 
= 
q_\ell
-
( \rho + P \circ \phi_{n+1} )\chi^2 ,
\end{align*} 
where $P$ has zero space average and is given implicitly by
\begin{align*}
\nabla P =
\frac{1}{\mu} \mathcal{Q}
\left(
\Pi_{\neq 0} 
\sum_{\xi \in \Lambda}  
\partial_t  \left({a}_{\xi}^2 \theta_{\xi}^2 \psi_{\xi}^2 \xi \right)
\right).
\end{align*}

Let us recover the expression for the Reynolds stress at level $n+1$. Denote for simplicity $G_n := (v_n+z_n) \otimes (v_n+z_n) + q_n Id - \mathring{R}_n$ and $G_\ell := G_n \ast \chi_\ell$. It holds
\begin{align*}
\partial_t v_{n+1} &+ \dvgphin \left( (v_{n+1}+z_{n+1}) \otimes (v_{n+1}+z_{n+1}) \right) 
+ \nablaphin q_{n+1}  - \Dphin v_{n+1}
\\
&= \nonumber
\underbrace{\left[\chi\partial_t w^{(p+c)}_{n+1} + \dvgphin(\tilde{w}_{n+1} \otimes (v_\ell+z_{n+1}) + (v_\ell+z_{n+1}) \otimes \tilde{w}_{n+1} ) - \Dphin \tilde{w}_{n+1} \right]}_{=\,linear\, error}
\\
&\quad+ \nonumber
\underbrace{\left[\chi^2 \dvgphin\left(w^{(p)}_{n+1} \otimes w^{(p)}_{n+1} +\mathring{R}_\ell\right) + \chi^2 \partial_t w^{(t)}_{n+1} + {\nablaphin(q_{n+1}-q_\ell)}\right]}_{=\,oscillation\, error}
\\
&\quad+ \nonumber
\underbrace{\left[ w^{(p+c)}_{n+1}\chi' + 2 w^{(t)}_{n+1} \chi \chi' + (1-\chi^2) \dvgphin \mathring{R}_\ell  \right]}_{=\,cut-off\, error}
\\
&\quad+ \nonumber
\underbrace{\left[ \dvgphin((v_\ell+z_\ell) \otimes (v_\ell+z_\ell) - ((v_n+z_n) \otimes (v_n+z_n)) \ast \chi_\ell)\right]}_{=\,mollification\, error\,I}
\\
&\quad+ \nonumber
\underbrace{\left[
\chi_\ell\ast\Dphien v_n - \Dphien v_\ell
+ 
\dvgphien G_\ell 
-
\left( \dvgphien G_n  \right) \ast \chi_\ell
 \right]}_{=\,mollification\, error\, II}
\\
&\quad+ \nonumber
\underbrace{\left[(\Dphien - \Dphin)\, v_\ell +\left(\dvgphin-\dvgphien\right)
G_\ell \right]}_{=\,flow\, error}
\\
&\quad+ \nonumber
\underbrace{\left[\dvgphin( \tilde{w}^{(p)}_{n+1} \otimes \tilde{w}^{(c+t)}_{n+1} + \tilde{w}^{(c+t)}_{n+1} \otimes \tilde{w}_{n+1}) \right]}_{=\,corrector\, error}
\\
&\quad+ \nonumber
\underbrace{\left[ \dvgphin( (z_{n+1}-z_\ell) \otimes (v_\ell+z_{n+1}) + (v_\ell+z_\ell) \otimes (z_{n+1}-z_\ell)) \right]}_{=\,corrector \, error\,II}.
\end{align*}

Then, the estimates for $\mathring{R}_{n+1}$ and $q_{n+1}$ descend as in \autoref{sec:convex_energy}. More precisely, first one proves \eqref{est2:deriv_qR} using the construction of $q_{n+1}$ and the equation satisfied by $\mathring{R}_{n+1}$, plus the estimates on the velocity and $\|z_{n+1}\|_{L^\infty_\mathfrak{t} C_x^2} \lesssim M \kappa_{n+1}^{-4} = M \lambda_{n+2}^{\alpha/4}$. Notice that the presence of $z_{n+1}$ is the reason why we have the additional factor $\lambda_{n+2}^{\alpha/4}$ in the $L^\infty_\mathfrak{t} C_x^1$ norm of $\mathring{R}_{n+1}$.

Then, the bounds \eqref{est2:RinL1small} and \eqref{est2:RinL1large} are recovered with the usual approach, with only minor differences:
\begin{itemize}
\item
The many factors $(n+1)$ in the iterative estimates clearly play no role, since $\delta_n$, $\lambda_n^{-1}$ decay exponentially fast.
Moreover, the additional $\lambda_{n+2}^{\alpha/4}$ in the $L^\infty_\mathfrak{t} C_x^1$ norm of $\mathring{R}_{n+1}$ does not affect the estimate in $L^2_\mathfrak{t} L^1_x$, since the former only enters in the mollification error and still $\ell^{1/4} \lambda_n^{20} \lambda_{n+1}^{\alpha/4} = \ell^{-\alpha/8} \lambda_n^{-5} \lesssim \delta_{n+2}^2$ can be made arbitrary small.
\item
The term involving the increment $z_{n+1}-z_\ell$ can be bound in $L^2_{\mathfrak{t},x}$ observing that
\begin{align*}
\|z_{n+1} - z_n \|_{L^2_\mathfrak{t}L^2_x}
&\lesssim
M \kappa_n^{-1}
=
M \lambda_{n+1}^{-\alpha/16}
\lesssim
\delta_{n+2}^2,
\\
\|z_n - z_\ell \|_{L^2_\mathfrak{t}L^2_x}
&\lesssim
\ell^{1/4}
\|z_n\|_{C^{1/4}_{\mathfrak{t},x}}
\lesssim
M \ell^{1/4} \kappa_n^{3}
\lesssim
\delta_{n+2}^2.
\end{align*}
\item
The new cut-off error
\begin{align*}
\mathcal{R}^{\phi_{n+1}}
\left( w^{(p+c)}_{n+1}\chi' + 2 w^{(t)}_{n+1} \chi \chi' + (1-\chi^2) \dvgphin \mathring{R}_\ell  \right)
\end{align*}
is such that the perturbations $w^{(p+c)}_{n+1}$ and $w^{(t)}_{n+1}$ are sufficiently small in $L^2_\mathfrak{t} L_x^p$, $p>1$ close to one, to compensate for the additional factor $2^n$ coming from $|\chi'|$. 
Finally, the remaining $\mathcal{R}^{\phi_{n+1}} (1-\chi^2) \dvgphin \mathring{R}_\ell = (1-\chi^2) \mathring{R}_\ell$ gives the factor $n+1$ for $t \leq \tau_{n}$.
\end{itemize}

To conclude, \eqref{est2:press} is deduced by arguments similar to those of \autoref{sec:convex_energy}.

\bibliographystyle{plain}

\begin{thebibliography}{10}

\bibitem{Ag22+}
Antonio Agresti.
\newblock Delayed blow-up and enhanced diffusion by transport noise for systems
  of reaction-diffusion equations.
\newblock arXiv:2207.08293, 2022.

\bibitem{BuCoVi18}
T.~Buckmaster, Maria Colombo, and Vlad Vicol.
\newblock Wild solutions of the {Navier–Stokes equations whose singular sets
  in time have Hausdorff} dimension strictly less than 1.
\newblock {\em Journal of the European Mathematical Society}, 24(9):3333–3378, 2022.

\bibitem{BuVi19b}
Tristan Buckmaster and Vlad Vicol.
\newblock Convex integration and phenomenologies in turbulence.
\newblock {\em EMS Surv. Math. Sci.}, 6(1/2):173–263, 2019.

\bibitem{BuVi19}
Tristan Buckmaster and Vlad Vicol.
\newblock {Nonuniqueness of weak solutions to the Navier-Stokes equation}.
\newblock {\em Ann. of Math.}, 189(1):101--144, 2019.

\bibitem{BuMoSz20}
Jan Burczak, Stefano Modena, and L{\'a}szl{\'o} Sz{\'e}kelyhidi.
\newblock Non uniqueness of power-law flows.
\newblock {\em Commun. Math. Phys.}, 388:199--243, 2020

\bibitem{DaSz17}
Sara {Daneri} and L{\'a}szl{\'o} {Sz{\'e}kelyhidi}.
\newblock {Non-uniqueness and h-Principle for H{\"o}lder-Continuous Weak
  Solutions of the Euler Equations}.
\newblock {\em Archive for Rational Mechanics and Analysis}, 224(2):471--514,
  May 2017.

\bibitem{DLS13}
Camillo De~Lellis and László Székelyhidi.
\newblock Dissipative continuous {E}uler flows.
\newblock {\em Invent. Math.}, 193:377--407, 2013.

\bibitem{DeHoVo16}
Arnaud Debussche, Martina Hofmanová, and Julien Vovelle.
\newblock Degenerate parabolic stochastic partial differential equations:
  Quasilinear case.
\newblock {\em The Annals of Probability}, 44(3):1916--1955, 2016.

\bibitem{DePa22+}
Arnaud Debussche and Umberto Pappalettera.
\newblock Second order perturbation theory of two-scale systems in fluid
  dynamics.
\newblock arXiv:2206.07775, 2022.

\bibitem{DeFlVi14}
Fran\c{c}ois Delarue, Franco Flandoli, and Dario Vincenzi.
\newblock Noise prevents collapse of {V}lasov-{P}oisson point charges.
\newblock {\em Communications on Pure and Applied Mathematics},
  67(10):1700--1736, 2014.

\bibitem{FlGaLu21c}
F.~Flandoli, L.~Galeati, and D.~Luo.
\newblock Delayed blow-up by transport noise.
\newblock {\em Comm. Partial Differential Equations}, 46, 2021.

\bibitem{FlGuPr10}
F.~Flandoli, M.~Gubinelli, and E.~Priola.
\newblock Well-posedness of the transport equation by stochastic perturbation.
\newblock {\em Invent. math.}, 180:1--53, 2010.

\bibitem{FlGuPr11}
F.~Flandoli, M.~Gubinelli, and E.~Priola.
\newblock Full well-posedness of point vortex dynamics corresponding to
  stochastic 2{D} {E}uler equations.
\newblock {\em Stochastic Processes and their Applications}, 121(7):1445--1463,
  2011.

\bibitem{FlLu21}
Franco Flandoli and Dejun Luo.
\newblock High mode transport noise improves vorticity blow-up control in 3{D}
  {N}avier–{S}tokes equations.
\newblock {\em Probab. Theory Relat. Fields}, 180:309--363, 2021.

\bibitem{FlPa21}
Franco Flandoli and Umberto Pappalettera.
\newblock 2{D} {E}uler equations with {S}tratonovich transport noise as a
  large-scale stochastic model reduction.
\newblock {\em J. Nonlinear Sci.}, 31:24, 2021.

\bibitem{FlPa22}
Franco Flandoli and Umberto Pappalettera.
\newblock From additive to transport noise in 2{D} fluid dynamics.
\newblock {\em Stoch. PDE: Anal. Comp.}, 10:964--1004, 2022.

\bibitem{HoLaPa22+}
Martina Hofmanová, Theresa Lange, and Umberto Pappalettera.
\newblock Global existence and non-uniqueness of 3{D E}uler equations perturbed
  by transport noise.
\newblock arXiv:2212.12217, 2022.

\bibitem{HoZhZh22b+}
Martina Hofmanová, Rongchan Zhu, and Xiangchan Zhu.
\newblock {Global-in-time probabilistically strong and Markov solutions to stochastic 3D Navier–Stokes equations: Existence and nonuniqueness}.
\newblock {\em The Annals of Probability}, 51(2):524--579, 2023.

\bibitem{La22+}
Theresa Lange.
\newblock Regularization by noise of an averaged version of the
  {N}avier-{S}tokes equations.
\newblock arXiv:2205.14941, 2022.

\bibitem{Lu21++}
Dejun Luo.
\newblock Regularization by transport noise for {3D MHD} equations.
\newblock arXiv:2107.00190, 2021.

\end{thebibliography}

\end{document}